\documentclass[a4paper,10pt]{amsart}
\usepackage{fullpage}
\usepackage[utf8]{inputenc}
\usepackage[T1]{fontenc}
\usepackage{amsmath,amsfonts,amssymb,amsthm}
\usepackage{pb-diagram,graphicx}
\usepackage{epsfig}
\usepackage{graphics}
\usepackage[pst-3d]{pstricks}
\usepackage{psfrag}
\usepackage{latexsym}
\usepackage{multicol}
\usepackage{times}
\usepackage{bbold}
\usepackage{dsfont}
\usepackage{mathrsfs}
\usepackage[foot]{amsaddr}
\usepackage{mathdots}
\usepackage{amsaddr}
\newcommand{\C}{\mathcal{C}}
\newcommand{\F}{\mathcal{F}}
\newcommand{\A}{\mathcal{A}}
\newcommand{\B}{\mathcal{B}}
\newcommand{\G}{\mathcal{G}}
\newcommand{\M}{\mathcal{M}}

\newcommand{\Om}{\Omega}
\newcommand{\ta}{\theta}
\newcommand{\R}{\mathds{R}}

\newcommand{\divv}{\operatorname{div}}
\newcommand{\dn}{\partial_n}

\newcommand{\po}{{\partial \Om}}
\newcommand{\VV}{\ta}
\newcommand{\Sb}{S}
\newcommand{\Tb}{T}
\newcommand{\Ub}{U}

\newcommand{\tr}{\operatorname{tr}}

\newcommand{\transp}{\mathsf{T}}

\newcommand{\cL}{\mathfrak{L}}
\newcommand{\vp}{\varphi}

\newcommand{\Rd}{\R^d}

\newcommand{\om}{\omega}

\newcommand{\tz}{s_0}
\newcommand{\Idd}{\operatorname{Id}}
\newcommand{\hold}{\mathcal{D}}
\newcommand{\Tt}{T_s}

\newcommand{\bproof}{\begin{proof}}
\newcommand{\eproof}{\end{proof}}

\newcommand{\hvp}{\hat\vp}

\newcommand{\iom}{\int_\Om}
\newcommand{\iomt}{\int_{\Om_s}}

\newcommand{\bl}{B}
\newcommand{\tp}{\therefore}
\newtheorem{theorem}{Theorem}
\newtheorem{proposition}{Proposition}

\newtheorem{lemma}{Lemma}

\newtheorem{example}{Example}
\theoremstyle{definition}
\newtheorem{definition}{Definition}
\theoremstyle{definition}

\theoremstyle{definition}
\newtheorem{remark}{Remark}
\theoremstyle{definition}

\title[An abstract Lagrangian framework for computing shape derivatives]{An abstract Lagrangian framework for computing shape derivatives}
\author[1]{Antoine Laurain\textsuperscript{1}}
\author[1]{Pedro T. P. Lopes\textsuperscript{1}}
\author[2]{Jean C. Nakasato\textsuperscript{2}}
\address[1]{Instituto de Matem\'atica e Estat\'istica, Universidade de S\~ao Paulo, Rua do Mat\~ao 1010, 05508-090, S\~ao Paulo, SP, Brazil}
\address[2]{Instituto de Ci\^encias Matem\'aticas e de Computa\c{c}\~ao,
	Universidade de S\~ao Paulo, Avenida Trabalhador S\~ao-Carlense, 400, S\~ao Carlos, SP, Brazil}
\email{laurain@ime.usp.br, pplopes@ime.usp.br, nakasato@ime.usp.br}
\subjclass[2010]{49Q10, 49Q12, 35Q93, 35R37 }

\begin{document}
\begin{abstract}
In this paper we study an abstract framework for computing shape derivatives of  functionals subject to PDE constraints.
We revisit the Lagrangian approach using the implicit function theorem in an abstract setting tailored for applications to shape optimization.
This abstract framework yields practical formulae to compute the derivative of a shape functional, the material derivative of the state, and the adjoint state. Furthermore, it allows to gain insight on the  duality between the material derivative of the state and the adjoint state.
We show several applications of our main result to the computation of distributed shape derivatives for problems involving linear elliptic, nonlinear elliptic, parabolic PDEs and distributions.
We compare our approach with other techniques for computing shape derivatives including the material derivative method and  the averaged adjoint method.
\end{abstract}
\maketitle
\keywords{\textbf{Keywords.} shape optimization; Lagrangian methods; distributed shape derivatives}
\section{Introduction}\label{sec:intro}
In shape optimization the most commonly encountered problems consist of the minimization of a cost functional with respect to a geometric variable, where the functional depends on the solution of a partial differential equation (PDE) or a system of PDEs.  
Variational inequality-constraints have also been considered but the literature on this topic is much more scarce, and the variational inequality is often regularized, which leads to a sequence of shape optimization problems with equality constraints. 
We refer to the recent works \cite{MR3878790,MR3562369,HintLaur11,LWY,Luft2020} and the references therein for more details on this topic.
Thus, we observe  that many shape optimization problems fit into the framework of optimization problems with equality constraints.

Lagrangian methods are an efficient way to tackle such problems. 
Their use  in shape optimization has been pioneered by Jean C\'ea in \cite{MR862783}
and was further developed by Delfour and Zol\'esio via a minimax formulation \cite{MR2731611}.
More recently, the averaged adjoint method (AAM) introduced in \cite{MR3374631} has proven to be a versatile and efficient Lagrangian-type approach to compute shape derivatives. 
Other approaches to compute shape derivatives include the material derivative method \cite{MR1215733}, 
the rearrangement method introduced in \cite{MR2434064}, and more recently \cite{MR3274847,MR3489049} for shape functionals defined  as the infimum of an integral energy, using convex analysis and Gamma-convergence techniques. 
We refer to \cite{MR3374631} for a detailed discussion and comparisons of these methods.
In view of the variety of possible approaches to compute shape derivatives, the  questions of the scope, ease of use and relations between these methods arise naturally.
For instance, the recently proposed AAM is fairly general but requires to study certain properties of the so-called averaged adjoint, which may lead to rather lengthy proofs.
In this paper we revisit the classical Lagrangian approach in shape optimization which allows us to answer some of these questions; we focus in particular on investigating the relations between the material derivative method, the classical Lagrangian approach and the AAM.

One viewpoint is particularly recurrent in  the shape optimization literature: 
it is often argued that, unlike the material derivative method, Lagrangian approaches allow to bypass the computation of the material or shape derivative of the state  (i.e.  the solution of the PDE constraint).
This aspect was emphasized in the pioneering work \cite{MR862783}, where the goal was to provide a fast formal method to compute shape derivatives.
In \cite{MR3374631}, the explicit  motivation for introducing the AAM is to bypass the differentiability of the control-to-solution operator for optimization problems with equality constraints.
Our investigation sheds a new light on this common opinion, indeed we show in a fairly general setting that
the adjoint equation and the material derivative equation are dual in a certain sense, and that, under the regularity assumptions guaranteeing the existence of the shape derivative of the cost functional, one also obtains  the existence of both the adjoint and the material derivative of the state.
The main origin of this opinion is probably related to the standard technique for computing shape derivatives of cost functionals, which is not based on a Lagrangian approach but consists in directly computing the shape derivative of the reduced cost functional using the chain rule and the shape derivative of the state.
The drawback of this approach is that in low-regularity scenarios it is common that the shape derivative of the cost functional can be computed even though the shape derivative of the state does not exist, see for instance the example in \cite{MR2434064}.
Thus, bypassing the computation of the shape derivative of the state using a Lagrangian approach seems to be advantageous in these situations.
Nevertheless, it turns out that in such scenarios the material derivative of the state usually exists, although it might have a low regularity;
this is due to the fact that the existence of the  shape derivative of the state requires stronger regularity assumptions than the existence of the material derivative.
In fact, our main result shows that, at least within the framework of the standard implicit function theorem,  the material derivative somehow comes ``for free'' when computing shape derivatives of cost functionals with PDE constraints. 
Thus, we conclude that the point of Lagrangian approaches is not really to bypass the computation of the material derivative,   but rather 
to provide a straightforward way to compute the adjoint and the shape derivative of the cost functional.

The implicit function theorem is a key ingredient for obtaining  first-order optimality conditions in the theory of optimization with  equality constraints and Lagrange multipliers, and the core idea behind our main result is to revisit the classical Lagrangian approach in shape optimization via this theorem.
In shape optimization, it is often  used to compute the material derivative and indirectly the shape derivative of the solution of a PDE or of an eigenvalue, but on a case by case basis, see for instance several applications in \cite{MR3791463}.
For instance in  the material derivative method \cite{MR1215733}, the material derivative of the state and then its shape derivative are calculated first, and subsequently the adjoint state is inferred, which allows to compute the shape derivative of the cost functional.
The formal  Lagrange method  of C\'ea \cite{MR862783} is also based on the assumption that the shape derivatives of the state and adjoint state exist, hence this approach sometimes fails when the regularity of the  data is too low, see the discussion and example in \cite{MR3374631}.

The main contribution of the present paper is to  provide  a more systematic approach in the form of an abstract framework for applying the implicit function theorem, within a Lagrangian setting, to compute derivatives of shape functionals in the case of equality constraints.
This allows to gain insight on the underlying structure connecting the material derivative of the state, the adjoint state and the shape derivative of the cost function, and to explore with more precision the question of minimal regularity assumptions. 
Also, in this paper we depart from certain standards of the shape optimization literature in two notable ways.
First of all, in our approach the focus is completely shifted to the material derivative instead of the shape derivative of the state.
Indeed, the material derivative of the state appears as a natural byproduct of the abstract Lagrangian framework rather than the shape derivative of the state. 
Second, in the various examples presented in this paper we systematically favor weak forms of the shape derivatives using a tensorial representation in the spirit of \cite{LAURAIN2020328,MR3535238}, rather than the usual boundary expressions, also known as Hadamard formulas.
This is motivated by the fact that the weak form of the shape derivative, also called  distributed shape derivative, requires in general  weaker regularity assumptions about the geometry, and that the corresponding boundary expression can immediately be inferred from the distributed shape derivative using the techniques of \cite{LAURAIN2020328,MR3535238}.
Another interesting feature of the paper is to present two examples of second-order tensor representation for distributed shape derivatives, thus giving concrete applications for the theoretical tensor representations of arbitrary order considered in \cite{MR3535238}.

In \cite{MR3374631} it is claimed that in many situations the
implicit function theorem is not applicable;
it is nevertheless not clear when exactly it can or cannot be used.
Through our main result and various examples and comparisons, the limits of applicability of the implicit function theorem and the purpose of the various methods described above for computing shape derivatives become much clearer.
Our main conclusion is that the Lagrangian approach based on  the standard implicit function theorem is relatively straightforward and easy to use, but is limited to the case where a strong material derivative exists.
We show that the hypotheses of the AAM \cite{MR3374631} and of the rearrangement method \cite{MR2434064} are more difficult to check but allow to work with weak material derivatives.
Still, we demonstrate that the abstract Lagrangian framework covers a broad range of relevant situations.
Its versatility is demonstrated here by the application to the calculation of distributed shape derivatives for several problems involving linear elliptic, nonlinear elliptic, linear parabolic PDEs, distributions, and cost functionals with higher-order derivatives.

The paper is organized as follows. 
In Section \ref{sec:lagrangian} we give several basic definitions and lemmas, and introduce the notations used in the paper.
Then, we describe the abstract Lagrangian setting and give the main result of the paper for computing shape derivatives within this framework.
In Section \ref{sec:bilinear}, we apply our main result to several linear and nonlinear elliptic problems, including   problems involving distributions and cost functionals with second-order derivatives, see Sections \ref{sec:distribution2} and \ref{sec:higher-order}, respectively.  
In Section \ref{sec:time}, we show applications involving parabolic problems.
In Section \ref{sec:aam}, we compare our approach with the AAM and the variational approach of \cite{MR2434064}.
We also provide an Appendix for technical results.
\pagebreak
\section{Abstract Lagrangian framework for shape optimization}\label{sec:lagrangian}
\subsection{General  notations and results}\label{section1a0}
We start by introducing notations for first, second and third-order tensors, and recall some basic rules of tensor calculus that will be used throughout the paper.
For sufficiently smooth  $\Om\subset\R^d$, $n$ denotes  the outward unit normal to $\Om$.
For vector-valued functions $a,b,c:\Om\to \R^d$ and second order tensors $\Sb:\Om\to\R^{d\times d}$ 
and $\Tb:\Om\to\R^{d\times d}$ whose entries are denoted by $\Sb_{ij}$ and $\Tb_{ij}$, the double dot product of $\Sb$ and $\Tb$ is defined as $\Sb : \Tb = \sum_{i,j = 1}^d \Sb_{ij}\Tb_{ij}$,
and the outer product $a\otimes b$ is defined as the second order tensor with entries $[a\otimes b]_{ij} = a_i b_j$.
For $\alpha\in\mathds{N}_{0}^{d}$ and $x\in\R^d$, we use the multi-index notation $x^{\alpha}=x_{1}^{\alpha_{1}}\dots x_{d}^{\alpha_{d}}$ and $\partial^{\alpha}=\frac{\partial^{|\alpha|}}{\partial x_{1}^{\alpha_{1}}\dots \partial x_{d}^{\alpha_{d}}}$,
with $|\alpha|=\alpha_{1}+\dots+\alpha_{d}$.
We also introduce the following notations:
\begin{itemize}
 \item $S^\transp$ for the transpose of $S$,
 \item $\tr$ for the trace of a matrix,
\item $I_d$ for the identity matrix in $\R^{d\times d}$,
 \item $\Idd: x\mapsto x$ for the identity in $\R^d$,
\item $Da$ for the Jacobian matrix of $a$,
 \item $D_\Gamma a := Da - (Da) n\otimes n$ for the tangential derivative on $\po$,
 \item $\divv_\Gamma a := \divv a - (Da) n\cdot  n$ for the tangential divergence on $\po$,
\end{itemize}

For third-order tensors $\mathbb{S}:\Om\to\R^{d\times d\times d}$ and $\mathbb{T}:\Om\to\R^{d\times d\times d}$ whose entries are denoted by $\mathbb{S}_{ijk}$ and $\mathbb{T}_{ijk}$, the triple dot product of $\mathbb{S}$ and $\mathbb{T}$ is defined as $\mathbb{S} \tp \mathbb{T}= \sum_{i,j,k = 1}^d \mathbb{S}_{ijk}\mathbb{T}_{ijk}$.
The outer product $a\otimes \Tb$ is defined as the third order tensor with entries $[a\otimes \Tb]_{ijk} = a_i \Tb_{jk}$.
Following the notations of \cite{qi2017transposes}, and using Einstein summation convention,
\begin{itemize}
 \item $\mathbb{S} c$ represents the matrix with entries  $\mathbb{S}_{ijk}c_k$,
 \item $\mathbb{S} bc$ represents the vector with entries  $\mathbb{S}_{ijk}b_j c_k$,
 \item $a\mathbb{S} bc$ represents the scalar  $a_i \mathbb{S}_{ijk}b_j c_k $, and notice that we have $a\mathbb{S} bc = (\mathbb{S} c)b \cdot a$. 
\end{itemize}
\begin{definition}\label{def:transpose_third}
Let  $\mathbb{S}\in\R^{d\times d\times d}$ and $\mathbb{T}\in\R^{d\times d\times d}$ be two third-order tensors satisfying
$$ a \mathbb{S} bc = b \mathbb{T} c a \quad\text{ for all } a, b,c\in\R^d.$$
Then we call $\mathbb{T}$ the transpose of $\mathbb{S}$ and we write $\mathbb{T} = \mathbb{S}^\transp$. 
Using indicial notations we have $\mathbb{T}_{ijk} = \mathbb{S}_{kij}$. 
\end{definition}
It can be shown that the transpose of $\mathbb{S}$ always exists and is unique; see
\cite[Proposition 3.1]{qi2017transposes}.
Note that in general we have $\mathbb{S}^{\transp\transp}\neq \mathbb{S}$ and $\mathbb{S}^{\transp\transp\transp} = \mathbb{S}$.

\begin{lemma}[Tensor calculus]\label{lemma_tensor}
For $\Om\subset \R^d$ open,  vector-valued functions $a,b,c,d:\Om\to \R^d$, second order tensors $\Sb,\Tb,\Ub:\Om\to\R^{d\times d}$ and third-order tensor $\mathbb{S}:\Om\to\R^{d\times d\times d}$, we have
\begin{itemize}
\item $\Sb : (a\otimes b) = a\cdot \Sb b = \Sb^\transp a \cdot  b = \Sb^\transp : (b\otimes a)$,
\item $\Sb(a\otimes b) = \Sb a \otimes b$\quad  and\quad $(a\otimes b) \Sb =  a \otimes \Sb^\transp b$, 
\item $(a\otimes b) c = (c\cdot b) a$,
\item $(a\otimes b) : (c\otimes d) = (a\cdot c) (b\cdot d)=(c\otimes b) : (a\otimes d) = c\cdot (a\otimes d)b$,
\item $\Sb \Tb: \Ub = \Tb: \Sb^\transp\Ub$,
\item $\mathbb{S}^{\transp} a: \Tb = \mathbb{S}\tp (a\otimes \Tb)$.
\end{itemize}
\end{lemma}

Let $E$ and $F$ be two Banach spaces. For bounded linear operators
we use the following notations:
\begin{itemize}
\item The set $\mathcal{L}(E,F)$ denotes the space of continuous linear
maps from $E$ to $F$.
\item The notation $E^{*}:=\mathcal{L}(E,\R)$ is used for the dual space
of $E$.
\item We call $A\in\mathcal{L}(E,F)$ an isomorphism if it is bijective.
Due to the closed graph theorem $A^{-1}\in\mathcal{L}(F,E)$.
\end{itemize}
\begin{definition}[Adjoint]\label{def:adjoint}
Let $E$ and $F$ be Banach spaces and suppose that $F$ is reflexive. 
We define the adjoint of $A\in\mathcal{L}(E,F^{*})$, denoted by $A^{*}\in\mathcal{L}(F,E^{*})$, as 
\[
\langle A^{*}g,h\rangle_{E^{*},E}=\langle Ah,g\rangle_{F^{*},F}.
\]
\end{definition}
Note that the usual definition of the adjoint requires $A^{*}\in\mathcal{L}(F^{**},E^{*})$, but since $F$ is reflexive we have $A^{*}\in\mathcal{L}(F,E^{*})$ in Definition \ref{def:adjoint}.
The specific choice of $A\in\mathcal{L}(E,F^{*})$ is motivated by applications to shape optimization problems, see Section~ \ref{sec:lag}.
The following property is a key ingredient for the main result of this paper, see Theorem~\ref{thm1}.
\begin{lemma}\label{lem-adjoint}
Let $E$ and $F$ be Banach spaces  and suppose that $F$ is reflexive. Then $A\in\mathcal{L}(E,F^{*})$
is an isomorphism if, and only if, $A^{*}\in\mathcal{L}(F,E^{*})$
is an isomorphism. 
\end{lemma}
\begin{proof}
It is well-known that $A:E\to F^{*}$ is an isomorphism if and only
if $A^*:F^{**}\to E^{*}$ is an isomorphism, see \cite[Chapter
3, Section 3.3]{kato2013perturbation}. Since we have assumed that $F$ is a reflexive
space, we have $F^{**}=F$. 
\end{proof}

\subsection{Shape calculus tools}\label{section1a}
In this section, we recall standard notations and basic results about perturbations of open sets using diffeomorphisms and  Eulerian shape derivatives.
Let $\mathds{P}(\hold)$ be the set of open sets compactly contained in $\hold$, where $\hold\subset \Rd$ is assumed to be open and bounded.
We define, for $k\geq 0$ and $0\leq \alpha \leq 1$,
\begin{align*}
\C^{k,\alpha}_c(\hold,\Rd) &:=\{\ta\in \C^{k,\alpha}(\hold,\Rd)\ |\ \ta\text{ has compact support in } \hold\},
\end{align*}
and $\C^k_c(\hold,\Rd)$, $\C^\infty_c(\hold,\Rd)$ in a similar way.
Consider a vector field $\ta\in \C^{0,1}_c(\hold,\Rd)$ and the associated flow
$\Tt^{\ta}:\hold\rightarrow \hold$, $s\in [0,s_1]$, defined for each $x_0\in \hold$ as $\Tt^{\ta}(x_0):=x(s)$, where $x:[0,s_1]\rightarrow \R^d$ solves 
\begin{align}\label{Vxt}
\begin{split}
\dot{x}(s)&= \ta(x(s))    \quad \text{ for } s\in [0,s_1],\quad  x(0) =x_0.
\end{split}
\end{align}
Since  $\ta\in \C^{0,1}_c(\hold,\Rd)$ we
have by Nagumo's theorem \cite{MR0015180} that, for each $s\in [0,s_1]$, the flow $\Tt^{\ta}$ is a homeomorphism from $\hold$ to $\hold$ and
maps boundary onto boundary and interior onto interior.
To compute shape derivatives, we only need to consider $T_s$ in an arbitrary small interval $[0,s_1]$.
For $\Om\in \mathds{P}(\hold)$, 
we consider the family of perturbed open sets  
\begin{equation}\label{domain}
\Om_s := \Tt^{\ta}(\Om)\in \mathds{P}(\hold). 
\end{equation}
Note that $\Om_{0}=T_{0}^\theta(\Om)=\Om$, as $T_{0}^\theta=\Idd$, and
we often write $\Tt$ instead of $\Tt^{\ta}$ for simplicity.

\begin{definition}[Eulerian shape derivative]\label{def1}
Let $J : \mathds{P}(\hold) \rightarrow \R$ be a shape functional.
\begin{itemize}
\item[(i)] The Eulerian semiderivative of $J$ at $\Om$ in direction $\ta \in \C^{0,1}_c(\hold,\Rd)$
is defined by, when the limit exists,
\begin{equation}
d J(\Om)(\ta):= \lim_{s \searrow 0}\frac{J(\Om_s)-J(\Om)}{s}.
\end{equation}
\item[(ii)] $J$ is said to be \textit{shape differentiable} at $\Om$ if it has a Eulerian semiderivative at $\Om$ for all $\ta \in \C^\infty_c(\hold,\Rd)$ and the mapping
\begin{align*}
d J(\Om): \C^\infty_c(\hold,\Rd) &  \to \R,\; \ta     \mapsto d J(\Om)(\ta)
\end{align*}
is linear and continuous, in which case $d J(\Om)(\ta)$ is called the \textit{Eulerian shape derivative} or simply \textit{shape derivative} at $\Om$.
\end{itemize}
\end{definition}

We now introduce several notations and well-known differentiability results which will be often used throughout the paper for the computation of shape derivatives.
\begin{definition}\label{def5}
Let $\Om\in \mathds{P}(\hold)$,  $\ta\in \C^1_c(\hold,\Rd)$ and consider the associated flow
$\Tt:\hold\rightarrow \hold$, $s\in [0,s_1]$.
Then we define the Jacobian $\xi(s) :=|\det DT_s|$ in $\hold$
and the tangential Jacobian $\xi_{\Gamma}(s) := |\det DT_s|\cdot |DT_s^{-\transp} n| $ on $\po$; see \cite[(4.11), p. 484]{MR2731611}.
For $s_1$ sufficiently small we have $\xi(s) =\det DT_s$  for all $s\in [0,s_1]$.
For a matrix $Q\in\R^{d\times d}$, we also define $\M(s,Q) := \xi(s)DT_s^{-1} Q DT_s^{-\transp}$.
\end{definition}
\begin{lemma}\label{lem01}
Let $\ta\in \C^k_c(\hold,\Rd)$, $k\geq 1$, and consider the associated flow
$\Tt:\hold\rightarrow \hold$, $s\in [0,s_1]$.
Let $\Om\in \mathds{P}(\hold)$ be Lipschitz and $n$ be the unit outward normal to $\Om$.
Then  
\begin{itemize}
 \item $s\mapsto\xi(s)\in \C^1([0,s_1], \C^{k-1}(\hold))$, $\xi(0)=1$  and  $\xi'(0) = \divv(\ta)$;
 \item $s\mapsto\xi_\Gamma(s)\in \C^1([0,s_1], \C^0(\po))$, $\xi_\Gamma(0)=1$  and
 $\xi'_\Gamma(0) = \divv_\Gamma(\ta): = \divv(\ta) - D\ta n\cdot n$;
 \item $s\mapsto T_s\in \C^1([0,s_1], \C^k(\hold, \R^{d}))$; 
 \item $s\mapsto DT_s\in \C^1([0,s_1], \C^{k-1}(\hold, \R^{d\times d}))$  with $DT_0 = I_d$ and $\partial_s DT_s|_{s=0} = D\ta$;
 \item  $s\mapsto DT_s^{-1}\in \C^1([0,s_1], \C^{k-1}(\hold, \R^{d\times d}))$  with $\partial_s DT_s^{-1}|_{s=0} = -D\ta$;
 \item If $k\geq 2$ then $s\mapsto D^2T_s\in \C^1([0,s_1], \C^{k-2}(\hold, \R^{d\times d\times d}))$;
 \item $\M'(0,Q) := \partial_s \M(0,Q) = \divv(\ta)Q - D\ta Q - Q D\ta^\transp$.
\end{itemize}
\end{lemma}
\begin{proof}
Applying \cite[Theorem 4.4, p. 189]{MR2731611} in the particular case $\ta\in \C^k_c(\hold,\Rd)$ for $k\geq 1$ we obtain that $s\mapsto T_s\in \C^1([0,s_1], \C^k(\hold, \R^{d}))$ which also yields the desired differentiability properties of  $s\mapsto DT_s$ and $s\mapsto D^2T_s$.
Since $DT_0 =I_d$ we obtain $s\mapsto DT_s^{-1}\in \C^1([0,s_1], \C^{k-1}(\hold, \R^{d\times d}))$ for $s_1$ sufficiently small.
The differentiability properties of $s\mapsto\xi(s)$ can be found for instance in \cite[Theorem 4.1, p. 482]{MR2731611}, and see \cite[Section 4.2, pp. 484-485]{MR2731611} for $s\mapsto\xi_\Gamma(s)$.
\end{proof}

\subsection{Abstract Lagrangian framework for shape optimization}\label{sec:lag}
Let $\Om\in \mathds{P}(\hold)$,   $E = E(\Om), F=F(\Om)$ be two Banach spaces with $F$ reflexive, and consider a parameterization 
$\Om_s$ defined by \eqref{domain} for $s\in [0,s_{1}]$ for some  $s_{1}>0$ and $\ta\in \C^{0,1}_c(\hold,\Rd)$. 
Shape optimization problems often consist in the minimization of a shape functional $J : \mathds{P}(\hold) \rightarrow \R$  depending on a state $u_s\in E(\Om_s)$ defined as the  solution of a partial differential equation (PDE), which can be seen as a minimization problem with an equality constraint.
It is then common to reformulate the constrained optimization problem as an unconstrained optimization problem with a linear penalization in the form of a Lagrangian functional $\cL$ satisfying $J(\Om_s) = \cL(\Om_s, u_s,\hat\psi)$,
where  $\hat\psi\in F(\Om_s) $.
The Lagrangian $\cL$ is composed of the cost functional and the variational formulation of the PDE. 
The particularity of shape optimization problems is that a pullback $\Psi_s$ is used in order to work on the fixed open set $\Om$ instead of $\Om_s$.
The pullback $\Psi_{s}:E(\Om_{s})\to E(\Om)$ or  $\Psi_{s}:F(\Om_{s})\to F(\Om)$  often corresponds to the precomposition $\Psi_{s}v(x):=v\circ T_{s}(x)$, where $T_{s}:\Om\to\Om_{s}$ is a diffeomorphism, but can also be a more complicated diffeomorphism, for instance when the spaces $H(\divv)$ and $H(\operatorname{curl})$ are involved; see \cite{MR3350625}.
Applying these pullbacks to the Lagrangian, one defines the so-called {\it shape-Lagrangian} as
$$ \G(s,\vp,\psi) :=  \cL(\Om_s, \Psi_s^{-1}(\vp),\Psi_s^{-1}(\psi))\ \text { for all } (s,\vp,\psi)\in [0,s_{1}]\times E(\Om)\times F(\Om).$$
This yields in particular
\begin{equation}\label{eq:JG}
 J(\Om_s) = \G(s,u^s,\psi)\ \text { for all }\psi\in F(\Om), 
\end{equation}
with $u^s :=\Psi_s(u_s)\in E(\Om)$.

In this paper we consider the following specific form for $\G:[0,s_{1}]\times E \times F \rightarrow \R$ for some $s_1>0$,
which is appropriate for shape optimization problems with equality constraints:
\begin{equation}
\label{G_lag}
\G(s,\vp,\psi):= \langle \A(s,\vp),\psi \rangle_{F^*,F}
  + \B(s,\vp),
\end{equation}
where
$$ \A:[0,s_{1}]\times E \rightarrow F^{*} \qquad \textrm{and} \qquad \B:[0,s_{1}]\times E \rightarrow \R .$$
Here,  $\B$ arises from the shape functional $J$ to be minimized, while $\A$ is derived from the variational formulation of the PDE via the pullback $\Psi_{s}$.

We can now state Theorem \ref{thm1}, the main result of this paper, whose relevance  is twofold.
The first feature is to give the practical formulae \eqref{dsG} for the shape derivative of the cost functional, \eqref{eq_material_der} for the material derivative of the state, and \eqref{adj_eq} for the adjoint state, with only a few natural conditions on the derivatives of $\A$ and $\B$ to be verified.
The second feature is to show the duality between the material derivative $\dot u$ and the adjoint $p$, in the sense that $\dot u\in E$ always exists under these natural conditions if the space $E$ is ``sufficiently large''; see the example of Section \ref{sec:distribution2}.
\begin{theorem}\label{thm1}
Let $\G:[0,s_{1}]\times E \times F \rightarrow \R$ be defined as in \eqref{G_lag} with $\A\in \C^{1}([0,s_1]\times E,F^{*})$, $u\in E$ be such that $\A(0,u)=0$ and $\B:[0,s_1]\times E\to\R$ be differentiable at $(0,u)$. 
Denote by $A(u):=\partial_{\vp}\A(0,u)\in\mathcal{L}(E,F^{*})$, $L(u) :=\partial_{s}\A(0,u)\in F^{*}$ and $\bl(u):=\partial_{\vp}\B(0,u)\in E^{*}$. 

Suppose that the linear operator $A(u):E\to F^{*}$ is an isomorphism. Then, there exists $s_{0}\in (0,s_{1}]$ and a unique $\C^{1}$
function $\left[0,s_{0}\right]\ni s \mapsto u^s \in E$ such that $u^{0}=u$
and $\A(s,u^{s})=0$ for all $s\in [0,s_{0}]$. 
Also, the  derivative $\dot u\in E$ of  $s\mapsto u^s$ at $s=0$ exists and $\dot u\in E$  is the unique solution of 
\begin{equation}\label{eq_material_der}
A(u)\dot{u}=-L(u) \in F^*.
\end{equation}
Now let $\ta\in \C^{0,1}_c(\hold,\Rd)$, $\Om_s$ defined in \eqref{domain}, and suppose that $J(\Om_s)$ can be written as in \eqref{eq:JG}.
Then under the above assumptions the shape derivative of $J(\Om)$ in direction $\ta$ is given by
\begin{equation}\label{dsG}
 dJ(\Om)(\ta) = \partial_s \G(0,u,p) =\langle L(u),p \rangle_{F^*,F} + \partial_s \B(0,u), 
\end{equation}
where the adjoint state $p\in F$ is the unique solution of 
\begin{equation}\label{adj_eq}
A(u)^{*}p=-\bl(u) \in E^*,
\end{equation}
and $A(u)^{*}: F\to E^*$ is the adjoint of  $A(u):E\to F^{*}$.
\end{theorem}
\begin{proof}
Since by assumption $\A(0,u)=0$ and  $\partial_{\vp}\A(0,u) = A(u): E\to F^*$ is an isomorphism, we can apply the implicit function theorem  \cite[Theorem 4.B, p. 150]{MR816732} and conclude that there exist $s_{0}\in (0,s_{1}]$, an open neighborhood $U\subset E$ that contains $u$ and a unique $\C^{1}$ function $[0,s_{0}]\ni s\mapsto u^{s}\in U$ such that $u^{0}=u$ and $\mathcal{A}(s,v)=0$ for $(s,v)\in[0,s_{0}]\times U$ if, and only if, $v=u^{s}$.

Since $s\mapsto u^{s}$ is $\C^{1}$, according to  \cite[Theorem 4.B, p. 150]{MR816732} we can take the following limit in the norm of $E$:
$$\lim_{s\searrow 0}\frac{u^s - u}{s} =\dot u = -  [\partial_\vp\A(0,u)]^{-1}\partial_s\A(0,u) = - A(u)^{-1}L(u) \in E,$$
where we have used 
\[
\A(s,u^{s})=0 \text{ for all } s\in [0,\tz] \implies\partial_{s}\A(0,u)+\partial_{\vp}\A(0,u) \dot u=0.
\]
This shows that $\dot u\in E$ is the unique solution of $A(u)\dot u=-L(u)\in F^*$.

Finally, in view of \eqref{eq:JG} and \eqref{G_lag}  we have, for $s\in[0,s_{0}]$,
\[
J(\Om_{s})=\G(s,u^{s},\psi)=\langle\A(s,u^{s}),\psi\rangle_{F^{*},F}+\B(s,u^{s}),\quad \forall\psi\in F.
\]
Let us assign $\psi =p$, where $p$ is the adjoint state defined by  \eqref{adj_eq}. Note that $A(u)^{*}$ is an isomorphism, as $A(u)$ is also an isomorphism, see Lemma \ref{lem-adjoint}; 
therefore the adjoint $p$ exists and is unique. 
Hence, using that $\B$ is differentiable at $(0,u)$, we get
\begin{align*}
dJ(\Om)(\theta)=\frac{d}{ds}J(\Om_{s})|_{s=0} & =\frac{d}{ds}\G(s,u^{s},p)|_{s=0}=\partial_{s}\G(0,u,p)+\langle\partial_{\vp}\G(0,u,p),\dot{u}\rangle_{E^{*},E}\\
 & =\partial_{s}\G(0,u,p)+\langle \partial_{\vp}\A(0,u)\dot{u}, p\rangle_{F^{*},F} +\langle \partial_{\vp}\B(0,u),\dot{u}\rangle_{E^{*},E}\\
 & =\partial_{s}\G(0,u,p) + \underbrace{\langle A(u)^{*}p,\dot{u}\rangle_{E^{*},E}+\langle \bl(u),\dot{u}\rangle_{E^{*},E}}_{=0 \text{ due to \eqref{adj_eq}}}\\
&  =\langle L(u),p\rangle_{F^{*},F}+\partial_{s}\B(0,u),
\end{align*}
which proves \eqref{dsG}.
\end{proof}
\begin{remark}
In the context of continuum mechanics and shape optimization, the solution $\dot u$  of \eqref{eq_material_der} is called {\it material derivative}, and  we will indeed refer to  $\dot u$ as the material derivative of $u$ in the rest of the paper. 
We observe that in applications the main quantity of interest is usually $dJ(\Om)(\ta)$ given by  \eqref{dsG}, which only requires the computation of the adjoint $p$ and not of $\dot u$. 
However,  $\dot u$  comes ``for free'' in Theorem~\ref{thm1} in the sense that the operators $L(u)$ and $A(u)$ appearing in \eqref{eq_material_der}  need to be computed anyway to get $dJ(\Om)(\ta)$ and the adjoint equation \eqref{adj_eq}.
\end{remark}
\begin{remark}
Depending on the choice of spaces $E$ and $F$, it may happen that either the adjoint equation \eqref{adj_eq} or the material derivative equation  \eqref{eq_material_der} in Theorem \ref{thm1} is formulated in a weaker form than expected, in the sense that $p$ or $\dot u$ may have lower regularity than what can be expected for a specific problem. 
Nevertheless, higher regularity for $p$ and $\dot u$ may sometimes be subsequently obtained  if $\bl(u)$ and $L(u)$ are in fact more regular than $\bl(u) \in E^*$ and $L(u) \in F^*$.
We will see instances of such situation when applying  Theorem \ref{thm1} in the next sections.  
\end{remark}

\section{Elliptic boundary value problems}\label{sec:examples}\label{sec:bilinear}
In this section we present several applications of Theorem \ref{thm1} to shape optimization problems involving elliptic equations in variational form. 
We show how these equations are naturally associated with a function $\mathcal{A}$ satisfying the assumptions of Theorem \ref{thm1}. 
In many applications, these assumptions arise naturally when considering coercive bilinear forms. 
In order to illustrate this, in Section \ref{sec:3.1}  we start by checking the hypothesis of Theorem \ref{thm1} in the case of elliptic equations of order $2\mu$ with Dirichlet conditions via G\aa rding's Theorem, and explain how to determine  $\mathcal{A}$. 
Then in Section \ref{sec:3.2} we explicitly compute the distributed shape derivative in the case of linear second order elliptic equations with Robin boundary conditions,
and  in the case of a second order quasilinear equation in Section  \ref{sec:3.3}.
In Sections \ref{sec:3.1} to \ref{sec:3.3} we present examples with $E=F$, so that the material derivative $\dot{u}$ and the adjoint state $p$ have the same regularity. 
In Sections \ref{sec:distribution2} and \ref{sec:higher-order} we consider two interesting examples where $E\neq F$, which results in different regularity properties for $\dot{u}$ and $p$.

\subsection{Elliptic equations of order $2\mu$}\label{sec:3.1}
In the general case of elliptic equations of order $2\mu$, we only check that the conditions of Theorem \ref{thm1} are satisfied, without computing explicitly the expression of the distributed shape derivative, as this would be beyond the scope of this paper.
Let $\Om\in\mathds{P}(\mathcal{D})$, $\ta\in \C^\mu_c(\hold,\Rd)$, the associated flow
$\Tt:\hold\rightarrow \hold$ and  $\Om_{s}:=T_{s}(\Om)$ be defined as in Section \ref{section1a}.
Let us consider the following equation for $u\in E(\Om_{s}) :=H_{0}^{\mu}(\Om_{s})$:
\begin{align}\label{eq:bvp}
\sum_{\left|\alpha\right|,\left|\beta\right|\le \mu}(-1)^{\left|\alpha\right|}\partial^{\alpha}\left(m_{\alpha\beta}\partial^{\beta}u\right)+\lambda u & =  f \quad  \textrm{ in }\Om_{s},
\end{align}
where $\sum_{\left|\alpha\right|,\left|\beta\right|=\mu}m_{\alpha\beta}(y)\xi^{\alpha+\beta}\ge c\left|\xi\right|^{2\mu}$,
for all $y\in\mathcal{D}$ , $\xi\in\R^{d}$ and for some
fixed constant $c>0$.  We assume that $m_{\alpha\beta}\in\C^{1}(\mathcal{D},\R)$
for all $\left|\alpha\right|$, $\left|\beta\right|\le \mu$, $\lambda \in \R$ and that $f\in H^{1}(\hold)$.

Associated with \eqref{eq:bvp}, we define a bilinear form
$c_{s}:E(\Om_{s})\times E(\Om_{s})\to\R$ by
\begin{equation}\label{cshigher}
c_{s}(\vp,\psi)=\sum_{\left|\alpha\right|,\left|\beta\right|\le \mu}\int_{\Om_{s}}m_{\alpha\beta}\partial^{\beta}\vp\partial^{\alpha}\psi+\lambda\int_{\Om_s}\vp\psi.
\end{equation} 
We assume that $\lambda$ is sufficiently large so that the bilinear form $c_{s}$ is coercive due to the G\aa rding's inequality \cite[Theorem 19.2]{wlokapartial},
that is, there is a constant $C>0$ depending only on $s$ such that
\[
c_{s}(\vp,\vp)\ge C\left\Vert \vp\right\Vert _{E(\Om_{s})}^{2},\,\forall\vp\in E(\Om_{s}).
\]

To define the variational formulation corresponding to \eqref{eq:bvp}, we introduce $\tilde{a}_{s}:E(\Om_{s})\times E(\Om_{s})\to\R$ as
\[
\tilde{a}_{s}(\vp,\psi) :=c_{s}(\vp,\psi)-\int_{\Om_{s}}f\psi,
\]
and $u_{s}$ is said to be a weak solution if $\tilde{a}_{s}(u_{s},\psi)=0,\,\forall\psi\in E(\Om_{s})$.

Let us denote $E:=E(\Om)$. Using the pullback $\Psi_{s}:E(\Om_{s})\ni v \mapsto v\circ T_{s} \in E$, we define $a_{s}:E\times E\to\R$ as
\begin{align*}
a_{s}(\vp,\psi) & :=  \tilde{a}_{s}(\Psi_{s}^{-1}(\vp),\Psi_{s}^{-1}(\psi))\\
 & =  \sum_{\left|\alpha\right|,\left|\beta\right|\le \mu}\int_{\Om_{s}}m_{\alpha\beta}\partial^{\beta}(\vp\circ T_{s}^{-1})\partial^{\alpha}(\psi\circ T_{s}^{-1}) +\lambda\int_{\Om_s}\vp\circ T_{s}^{-1}\psi\circ T_{s}^{-1}
   -\int_{\Om_{s}}f\psi\circ T_{s}^{-1}\\
 & =  \sum_{\left|\alpha\right|,\left|\beta\right|\le \mu}\int_{\Om}m_{\alpha\beta}^{s}\partial^{\beta}\vp\partial^{\alpha}\psi
 -\int_{\Om}f\circ T_{s}\psi\xi(s).
\end{align*}
The coefficients $m_{\alpha\beta}^{s}$ depend on spatial derivatives of order up to $\mu$ of $T_{s}$ and on the coefficients $m_{\alpha\beta}$ and $\lambda$. Hence they are continuous functions with continuous derivatives in $s$.

Due to the bilinearity there exists a unique continuous
linear operator $\mathcal{A}_0(s):E\to E^{*}$ such
that 
\[
\left\langle \mathcal{A}_0(s)\vp,\psi\right\rangle _{E^{*}\times E}
:=\sum_{\left|\alpha\right|,\left|\beta\right|\le \mu}\int_{\Om}m_{\alpha\beta}^{s}\partial^{\beta}\vp\partial^{\alpha}\psi.
\]
The linear functional $\mathcal{A}_1(s)\in E^{*}= H^{-\mu}\left(\Om\right)$ defined by
\[
\langle\mathcal{A}_1(s),\psi\rangle_{E^*,E} := \int_{\Om}f\circ T_{s}\psi\xi(s)
\]
is also continuous.
Gathering all these informations, we conclude that $\mathcal{A}(s,\vp):=\mathcal{A}_0(s)\vp-\mathcal{A}_1(s)$ belongs to $\C^{1}\left([0,s_{1}]\times E,E^{*}\right)$ and  is the unique function such that
\[
a_{s}(\vp,\psi)=\left\langle \mathcal{\mathcal{A}}(s,\vp),\psi\right\rangle _{E^{*} \times E}.
\]
By the linearity of $\mathcal{A}_0(s)$ with respect to $\vp$, we
readily conclude that $\partial_{\vp}\mathcal{A}:[0,s_{1}]\times E\to\mathcal{L}(E,E^{*})$
is given by $\partial_{\vp}\mathcal{\mathcal{A}}(s,\vp)=\mathcal{A}_0(s)$.
In particular, for $u=u_0$, we have that $A(u):=\partial_{\vp}\mathcal{\mathcal{A}}(0,u)=\mathcal{A}_0(0)$
is associated with the coercive form $c_{s}$ for $s=0$ as defined in \eqref{cshigher}. 
Then it follows by the Lax-Milgram theorem that $A(u):E\to E^{*}$
is an isomorphism.

We conclude that, for $E=F$, the function $\mathcal{A}$
satisfies all conditions of Theorem \ref{thm1}.
The arguments we have given can be  further  generalized for boundary
conditions that satisfy the so-called Agmon conditions \cite{wlokapartial}; the ideas
remain the same.

\subsection{Second order linear equations.}\label{sec:3.2}
Let $\Om\in\mathds{P}(\mathcal{D})$ be of class $\C^{1}$, $\ta\in \C^1_c(\hold,\Rd)$, the associated flow
$\Tt:\hold\rightarrow \hold$ and  $\Om_{s}:=T_{s}(\Om)$ be defined as in Section \ref{section1a}.
Assume that $f\in H^{1}(\hold)$, $g\in H^{2}(\hold)$ and $\beta\in \C^{1}(\hold)$ satisfies $\beta(y)>0$ for all $y\in\mathcal{D}$; these relatively strong assumptions are needed in order to apply Theorem \ref{thm1}.
Let $u_s\in H^{1}(\Om_{s})$ be the solution to
\begin{align}\label{Rob1}
\begin{split}
-\sum_{i,j=1}^{d}m_{ij}\partial_{i}\partial_{j}u_s & = f \quad \textrm{ in } \Om_{s},\\
\sum_{i,j=1}^{d}n_{i}m_{ij}\partial_{j}u_s   +   \beta u_s & = g  \quad \textrm{ on } \partial\Om_{s},
\end{split}
\end{align}
where $n$ is the outward unit normal vector
to $\Om_{s}$. 
The matrix $M=(m_{ij})\in \mathds{R}^{d\times d}$ is a positive definite symmetric matrix with constant coefficients. 
We can associate the following function $\tilde{a}_{s}:H^{1}(\Om_{s})\times H^{1}(\Om_{s})\to\R$ and bilinear form $c_{s}:H^{1}(\Om_{s})\times H^{1}(\Om_{s})\to\R$ to Problem \eqref{Rob1}:
\[
\tilde{a}_{s}\left(\vp,\psi\right) :=c_{s}(\vp,\psi)-\int_{\Om_{s}}f\psi-\int_{\partial\Om_{s}}g\psi
\quad \text{ and }\quad 
c_{s}(\vp,\psi) :=\int_{\Om_{s}}M\nabla\vp  \cdot  \nabla\psi+\int_{\partial\Om_{s}}\beta\vp\psi.
\]
A weak solution of \eqref{Rob1} is a function $u_s\in H^1(\Om_s)$ such that $\tilde{a}_s(u_s,\psi)=0$ for all $\psi\in H^1(\Om_s)$. The bilinear form $c_{s}$ is continuous and coercive \cite[Satz 7.37]{arendt2010partielle}.
Using the pullback $\Psi_{s}:H^{1}(\Om_{s})\ni v \mapsto v\circ T_{s} \in H^{1}(\Om)$, and defining  $\beta^{s}:=\beta\circ T_{s}$, $f^{s}:=f\circ T_{s}$, $g^{s}:=g\circ T_{s}$,
we introduce $a_{s}:H^{1}(\Om)\times H^{1}(\Om)\rightarrow\R$ as 
\begin{align*}
a_{s}(\vp,\psi) & :=\tilde{a}_{s}(\Psi_{s}^{-1}(\vp),\Psi_{s}^{-1}(\psi))\\
 & =\int_{\Om_{s}}M\nabla(\vp\circ T_{s}^{-1})\cdot\nabla(\psi\circ T_{s}^{-1})-f(\psi\circ T_{s}^{-1})
  +\int_{\partial\Om_{s}}(\beta\vp\circ T_{s}^{-1}-g)\psi\circ T_{s}^{-1}\\
 & =\left\langle \mathcal{A}_0(s)\vp - \mathcal{A}_1(s),\psi\right\rangle _{H^{1}(\Om)^{*}\times H^{1}(\Om)},
\end{align*}
with the linear operators $\mathcal{A}_0(s):H^{1}(\Om)\to H^{1}(\Om)^{*}$ and $\mathcal{A}_1(s)\in H^{1}(\Om)^{*}$ defined by
\begin{align*}
\left\langle \mathcal{A}_0(s)\vp,\psi\right\rangle _{H^{1}(\Om)^{*}\times H^{1}(\Om)}
& :=\int_{\Om} \mathcal{M}(s,M)  \nabla\vp\cdot\nabla\psi
+\int_{\partial\Om}\beta^{s}\vp\psi \, \xi_{\Gamma}(s),\\
\left\langle \mathcal{A}_1(s),\psi\right\rangle _{H^{1}(\Om)^{*}\times H^{1}(\Om)}
& :=\int_{\Om}f^{s}\psi\xi(s)+\int_{\partial\Om}g^{s}\psi\xi_{\Gamma}(s),
\end{align*}
see Definition \ref{def5} for the definition of $\mathcal{M}(s,M)$.
The smoothness assumptions allow us to define  $\mathcal{A}\in \C^{1}([0,s_{1}]\times H^{1}(\Om),H^{1}(\Om)^{*})$
by 
\[
\mathcal{\mathcal{A}}(s,\vp) :=\mathcal{A}_0(s)\vp-\mathcal{A}_1(s).
\]
The weak solution of  \eqref{Rob1} for $s=0$ corresponds to the unique function $u$ solving $\mathcal{A}(0,u)=0$.
The linearity of $\mathcal{A}_0(s)$ implies that $\partial_{\vp}\mathcal{A}:[0,s_{1}]\times H^{1}(\Om)\to\mathcal{L}(H^{1}(\Om),H^{1}(\Om)^{*})$
is given by $\partial_{\vp}\mathcal{\mathcal{A}}(s,\vp)=\mathcal{A}_0(s)$.
In particular, $A = A(u):=\partial_{\vp}\mathcal{\mathcal{A}}(0,u)=\mathcal{A}_0(0)$, with $u=u_0$,
is associated with the coercive form $c_{0}$ and the Lax-Milgram theorem
implies that $A:H^{1}(\Om)\to H^{1}(\Om)^{*}$ is an isomorphism. We conclude that for $E=F=H^{1}(\Om)$, the function $\mathcal{A}$
satisfies all conditions of Theorem \ref{thm1}.

Using $g\in H^2(\hold)$ and \cite[Lemma 5.3.9]{MR3791463} we obtain that $s\mapsto g^s$ is differentiable in $H^1(\hold)$;
then, $s\mapsto g^s$ is differentiable in $H^{1/2}(\partial\Om)$  using the fact that the trace operator is a bounded linear operator from $H^1(\hold)$ to $H^{1/2}(\partial\Om)$ since $\Om$ is $\C^1$.
Using then $\beta\in \C^{1}(\hold)$ and Lemma \ref{lem01}, this proves the differentiability of $s\mapsto \int_{\partial\Om}\beta^{s}\vp\psi \, \xi_{\Gamma}(s)$ and $s\mapsto\int_{\partial\Om}g^{s}\psi\xi_{\Gamma}(s)$.
The derivative $L(u) :=\partial_{s}\mathcal{A}(0,u)\in F^{*}$  is computed as follows (see Lemma \ref{lem01}):
\begin{align*}
\left\langle L(u),\psi\right\rangle_{F^{*},F} 
&:= \left\langle\partial_{s}\mathcal{A}(0,u),\psi\right\rangle _{F^{*},F} = \partial_s a_s(u,\psi)|_{s=0}  \\
& =  \iom \mathcal{M}'(0,M)\nabla u\cdot\nabla \psi 
- \iom \divv(f\ta)\psi  
+ \int_{\po} ( \beta u - g)\psi \divv_\Gamma(\ta)  + \psi(u\nabla \beta - \nabla g)\cdot\ta .
\end{align*}  
We have seen that  $A=\mathcal{A}_0(0)$, therefore
\begin{align*}
\langle A\hvp,\psi\rangle_{F^*,F} =
\langle\mathcal{A}_0(0)\hvp,\psi\rangle_{F^*,F}  & 
= \iom M \nabla\hvp\cdot\nabla \psi 
+ \int_{\po} \beta \hvp \psi .
\end{align*}

Consider the following cost functional $J:\mathds{P}(\mathcal{D})\to \R$ defined by
\begin{equation}\label{cost_Robin}
 J(\Om) := \frac{1}{2}\int_{\Om} |\nabla u|^2. 
\end{equation}
Using the pullback $\Psi_s$, we define the auxiliary function $\mathcal{B}:[0,s_1]\times E\to\mathds{R}$ associated with the cost functional \eqref{cost_Robin} as
$$ \B(s,\vp) :=  \frac{1}{2} \int_\Om D\Tt^{-\transp}\nabla \vp \cdot D\Tt^{-\transp}\nabla \vp  \xi(s).$$
We have $\mathcal{B}\in \C^1([0,s_1]\times E, \mathds{R})$ thanks to Lemma \ref{lem01}. 
Then we compute
$$
\langle \bl(u),\hvp\rangle_{E^*,E}
:=
\langle\partial_\vp \B(0,u),\hvp\rangle_{E^*,E}
= \iom \nabla u \cdot \nabla\hvp . $$
Note that $u\in E= H^1(\Om)$, therefore $\bl(u)\in E^* = H^1(\Om)^*$ is well defined.

Gathering these results, we can apply Theorem \ref{thm1} and we get 
the adjoint state $p\in E=F$ which is the unique solution of 
$$\langle A(u)^*p,\hvp\rangle_{E^*,E} = - \langle \bl(u),\hvp\rangle_{E^*,E}\quad \text{ for all }\hvp\in E.$$
Explicitly, this corresponds to 
\begin{align*}
 \iom M \nabla p \cdot\nabla \hvp
+ \int_{\po} \beta p \hvp 
& 
= -\iom \nabla u \cdot \nabla\hvp \quad \text{ for all }\hvp\in E.
\end{align*}

Finally, Theorem \ref{thm1} and Lemma \ref{lem01} yield the following expression of the shape derivative:
\begin{align*}
 dJ(\Om)(\ta) &= \langle L(u),p \rangle_{F^*,F} + \partial_s \B(0,u)\\
 &= \iom \M'(0,M)\nabla u\cdot\nabla p
- \iom \divv(f\ta)p 
+ \int_{\po} ( \beta u - g)p\divv_\Gamma(\ta)  
+ p (u\nabla \beta - \nabla g)\cdot\ta\\
&\quad +\iom \frac{1}{2} |\nabla u|^2\divv(\ta)  
- \nabla u\cdot D\ta^\transp \nabla u .
\end{align*}
Using the tensor relations of Lemma \ref{lemma_tensor} yields the following result.
\begin{proposition}
Let $\Om\in\mathds{P}(\mathcal{D})$ be of class  $\C^1$, $\ta\in \C^1_c(\hold,\Rd)$, $f\in H^{1}(\hold)$, $g\in H^{2}(\hold)$ and $\beta\in \C^{1}(\hold)$ satisfy $\beta(y)>0$ for all $y\in\mathcal{D}$.
Then the shape derivative of the cost functional \eqref{cost_Robin} subject to the constraint \eqref{Rob1} is given by
$$ dJ(\Om)(\ta) = \iom S_0\cdot \ta + S_1 : D\ta  + \int_{\po} S_{0,\Gamma}\cdot \ta + S_{1,\Gamma} : D_\Gamma\ta ,$$
with  $S_0\in L^1(\Om, \R^{d})$, $S_1\in L^1(\Om, \R^{d\times d})$,  $S_{0,\Gamma}\in L^1(\po, \R^{d})$ and $S_{1,\Gamma}\in L^1(\po, \R^{d\times d})$ given by
\begin{align*}
S_0 & = -p\nabla f,\\
S_1 & = - \nabla p\otimes M\nabla u  
-  \nabla u\otimes M \nabla p 
-  \nabla u\otimes \nabla u
+ \left[M\nabla u\cdot \nabla p - fp + \frac{1}{2} |\nabla u|^2\right] I_d,  \\
S_{0,\Gamma} & = p (u\nabla \beta - \nabla g),\\
S_{1,\Gamma} & = [(\beta u - g)p] I_d.
\end{align*}
\end{proposition}

\subsection{Second order quasilinear equations}\label{sec:3.3}
In this section $d$ is either $2$ or $3$. We suppose that
the set $\Om\in\mathds{P}(\mathcal{D})$ is Lipschitz and convex,  $\ta\in \C^1_c(\hold,\Rd)$ and  the associated flow  $T_s = T_s^\ta:\hold \to\hold$ and $\Om_s :=T_s(\Om)$ are defined as in Section \ref{section1a}. 
Here we consider the following non-linear problem:
\begin{align}\label{quasilinear}
	\begin{split}
		- \divv\left(m(y,u_s)\nabla u_s\right)+f(y,u_s) & =g\quad  \textrm{ in }\Om_{s},\\
		u_s & =0\quad \textrm{ on }\po_{s}.
	\end{split}
\end{align}
Our assumptions are that $g\in H^1(\mathcal {D})$,  $m,f:\mathcal{D}\times \R\to \R$ are of class $\C^{2}$ and there exist constants $c_{1}, c_{2}, c_{3}>0$ such that 
\begin{align}\label{cond_nonlin}
c_{1}\le m(y,r),\  c_{2}\le\min\left\{\partial_r f(y,r),\partial_r m(y,r)\right\},
\ \max\left\{\partial_r f(y,r),\partial_r m(y,r)\right\}\le c_{3},\forall(y,r)\in\mathcal{D}\times \R. 
\end{align}
If $s=0$, then $\Om_0=\Om$ and the assumptions imply that \eqref{quasilinear} has a unique solution $u_0$, see \cite[Theorem 2.1]{casas2008optimality}, and  $u_0\in H_0^1(\Om)\cap H^2(\Om)$, see  \cite[Corollary 2.5]{casas2008optimality}.

The variational formulation corresponding to  \eqref{quasilinear}  may be written as
\[
\tilde{a}_{s}(u_s,\psi) 
:=
\int_{\Om_s} m(y,u_s)\nabla u_s \cdot \nabla \psi
+ \int_{\Om_s} f(y,u_s)\psi
- g\psi
=0 \quad \text{ for all }\psi\in H^1_0(\Om_s).
\]
In order to work with variables in the fixed space  $H^1_0(\Om)$ we use the pullback  $\Psi_s : H^1_0(\Om_s)\to H^1_0(\Om), v\mapsto v\circ T_s$; see Section \ref{sec:lag}.
Let $E=F=H_0^1(\Om)$ and  $E^* = F^*=H^{-1}(\Om)$. We define $\A:[0,s_{1}]\times E \rightarrow F^{*}$ as
\begin{align*}
	\langle \A(s,\vp),\psi \rangle_{F^*,F} & :=\tilde{a}_{s}(\Psi_s^{-1}(\vp),\Psi_s^{-1}(\psi))\\
	&= \int_{\Om_s} m(y,\vp\circ T_s^{-1})\nabla (\vp\circ T_s^{-1})\cdot \nabla (\psi\circ T_s^{-1})
	+\int_{\Om_s} f(y,\vp\circ T_s^{-1})\psi\circ T_s^{-1}- g\psi\circ T_s^{-1}.
\end{align*}
Performing the change of variables $x\mapsto T_s(x)$ and using the chain rule we get
\begin{align}\label{asnonlin}
	\langle \A(s,\vp),\psi \rangle_{F^*,F} 
	& =\int_{\Om}m(T_s(x),\vp)\M(s,I_d)\nabla \vp \cdot\nabla \psi 
	+ f(T_s(x),\vp)\psi\xi(s)
	- g^s\psi\xi(s),
\end{align}
where $g^s := g\circ T_s$, and $\M(s,I_d)$ is defined in Definition \ref{def5}. 
Notice that $\A$ is $\C^1$; indeed by Lemma \ref{lem01} and the fact that $f$ is $\C^2$ and $g\in H^{1}(\mathcal{D})$, it follows that $\A$ is differentiable at $(0,u):=(0,u_0)$. In particular, we have
\begin{align*}
		\langle A(u)\hat{\vp},\psi\rangle_{F^*,F}
		&:= \langle \partial_\vp\A(0,u)(\hat\vp ),\psi \rangle_{F^*,F} 
		= \int_{\Om}\partial_r m(x,u)\hat\vp \nabla u\cdot\nabla \psi
		+ m(x,u)\nabla\hat\vp \cdot \nabla\psi 
		+ \partial_r f(x,u)\psi\hat{\vp}.
\end{align*}
In view of  \cite[Remark 2.7]{casas2008optimality}, $A(u): H_0^1(\Om)\to H^{-1}(\Om)$ is an isomorphism.

Let $u_d\in H^1(\hold)$ and introduce the tracking-type cost functional $J:\mathds{P}(\mathcal{D})\to \R$ as
\begin{equation}\label{cost_quasilinear}
J(\Om) :=\dfrac{1}{2}\int_{\Om} (u-u_d)^2. 
\end{equation}
We introduce the corresponding perturbed cost functional
$$
\B(s,\vp) :=\frac{1}{2}\int_{\Om}(\vp-u_d\circ T_s)^2\xi(s)
$$
and compute
$$
\langle \bl(u),\hat{\vp}\rangle_{E^*,E} 
:=\partial_\vp \B(0,u)(\hat\vp)
= \int_\Om\hat{\vp}(u-u_d).
$$
Then $\bl(u)\in E^*=H^{-1}(\Om)$ since $u\in H_0^1(\Om)\cap H^2(\Om)$ and  $u_d\in H^1(\hold)$. 
We also have using \eqref{asnonlin} and Lemma~\ref{lem01}:
\begin{align*}
	\langle L(u),\psi\rangle_{F^*,F}
	:=\left\langle\partial_{s}\mathcal{A}(0,u),\psi\right\rangle _{F^{*},F}
	&=\int_\Om  m(x,u)\M'(0,I_d)\nabla u\cdot \nabla\psi 
	+ \nabla_x m(x,u)\cdot\ta \nabla u\cdot \nabla\psi\\
	&\quad +\int_\Om f(x,u)\psi \divv\ta
	+ \nabla_x f(x,u)\cdot\ta \psi
	-\divv(g\ta)\psi,
\end{align*}
The function $L(u)$ belongs to $F^*$ due to the regularity assumptions. Indeed, since $u\in H_0^1(\Om)\cap H^2(\Om)$ and  $d=2$ or $d=3$, by  \cite[Theorem 3.8, p. 66]{necas}\footnote{Theorem 3.8 of \cite[p. 66]{necas} says that if $\Om$ is Lipschitz, $p\geq 1$, $kp>n$ and $\mu=k-(n/p)$ if $k-n/p<1$, $\mu<1$ if $k-n/p=1$ or $\mu=1$ if $k-n/p>1$, then $W^{k,p}(\Om)\subset C^{0,\mu}(\overline{\Om})$ algebraically and topologically.}  we have that $u\in\C^{0,\mu}(\overline{\Om})$ for all $\mu<1$ if $d=2$ and $u\in\C^{0,\frac{1}{2}}(\overline{\Om})$ if $d=3$. Thus, with the fact that $m$ is $\C^2$, we get that $m(x,u)\M'(0,I_d)\nabla u$ is in $H^1(\Om)$. 
We point out that the other terms are in $L^2(\Om)$ by analogous arguments.

Also, the partial derivative of $\B$ with respect to $s$ at $(0,u)$ exists thanks to Lemma \ref{lem01} and the regularity of $u_d$, thus we are in a position to apply Theorem \ref{thm1}. 
We obtain the existence of the material derivative $\dot u\in F$ satisfying
$$
\langle A(u)\dot u,\psi\rangle_{F^*,F}=-\langle L(u),\psi\rangle_{F^*,F}.
$$
The shape derivative has the expression
\begin{align*}
	dJ(\Om)(\ta) &=\langle L(u),p\rangle_{F^*,F}
	+\partial_s \B(0,u)\\
	& = \int_\Om  m(x,u)\M'(0,I_d)\nabla u\cdot \nabla p 
	+ \nabla_x m(x,u)\cdot\ta \nabla u\cdot \nabla p\\
	&\quad +\int_\Om f(x,u)p \divv\ta
	+ p\nabla_x f(x,u)\cdot\ta 
	-\divv(g\ta)p
	+\frac{1}{2}(u - u_d)^2\divv(\ta) 
	- \nabla u_d\cdot \ta (u-u_d),
\end{align*}
where the adjoint $p\in F=H_0^{1}(\Om)$ is the unique solution of 
$$
\langle A(u)^*p,\hat\vp\rangle_{E^*,E}=-\langle \bl(u),\hat{\vp}\rangle_{E^*,E}\quad \text{ for all }\hvp\in E.
$$
The adjoint equation can be written explicitly as
$$
\int_{\Om} m(x,u)\nabla p \cdot \nabla \hat\vp
+ \partial_r m(x,u)\hat\vp \nabla u\cdot\nabla p
+ \partial_r f(x,u) p \hat\vp
= - \iom\hat{\vp}(u-u_d)\quad \text{ for all }\hvp\in E.
$$

Finally, using the tensor relations of Lemma \ref{lemma_tensor}  we obtain the following result.
\begin{proposition}\label{prop2}
Let $d=2$ or $3$, $\Om\in\mathds{P}(\mathcal{D})$  Lipschitz and convex, $\ta\in \C^1_c(\hold,\Rd)$, $m,f:\mathcal{D}\times \R\to \R$  of class $\C^{2}$ and there exist constants $c_{1}, c_{2}, c_{3}>0$ such that 
\eqref{cond_nonlin} holds.
Then, the distributed shape derivative of the cost functional \eqref{cost_quasilinear} subject to the constraint \eqref{quasilinear} is given by
$$ dJ(\Om)(\ta) = \iom S_0\cdot \ta + S_1 : D\ta,$$
with $S_0\in L^1(\Om, \R^{d})$, $S_1\in L^1(\Om, \R^{d\times d})$,  defined as 
\begin{align*}
	S_0 & = (\nabla u\cdot\nabla p)\nabla_xm(x,u) 
	+ p\nabla_x f(x,u) - p\nabla g  -(u-u_d)\nabla u_d,\\
	S_1 & =  -m(x,u)(\nabla p\otimes \nabla u  
	+  \nabla u\otimes\nabla p) 
	+ \left[m(x,u)  \nabla u\cdot \nabla p + f(x,u)p -gp +\frac{1}{2}(u -u_d)^2\right] I_d.  
\end{align*}
\end{proposition}

\subsection{Shape derivatives for elliptic problems involving distributions}\label{sec:distribution2}
In Sections \ref{sec:3.1} to \ref{sec:3.3}, we have encountered situations where $E=F$. 
In fact, in Sections \ref{sec:3.2} and \ref{sec:3.3} the spaces were respectively $H^{1}(\Om)$ and $H_{0}^{1}(\Om)$, which are natural spaces when considering second order elliptic equations.
However, when the right-hand side of the PDE is less regular than the regularity considered in Sections \ref{sec:3.2} and \ref{sec:3.3}, the condition $\A\in \C^{1}([0,s_1]\times E,F^{*})$ of Theorem \ref{thm1} fails  if one tries to work with the spaces $H^{1}(\Om)$ or $H_{0}^{1}(\Om)$.
A remedy is to work with the larger space $E=L^{2}(\Om)$, 
the counterpart being that $F$ must then be a subspace of  $H^{1}(\Om)$ or $H^1_0(\Om)$. 
An interesting  consequence is that the material derivative is then only in $L^2(\Om)$, while the adjoint is more regular than $H^1(\Om)$.
To be more precise, in this subsection we study  Poisson's equation with Dirichlet condition, with $E=E^*=L^{2}(\Om)$, and $F=H^{2}(\Om)\cap H_{0}^{1}(\Om)$;
note that in Section \ref{sec:higher-order} we study the converse case where $F$ is equal to  $L^2(\Om)$ and $E$ is a subspace of $H^1_0(\Om)$.

Let $\Om\in\mathds{P}(\mathcal{D})$ be of class  $\C^{2}$, $\ta\in \C^2_c(\hold,\Rd)$, the associated flow
$\Tt:\hold\rightarrow \hold$ and  $\Om_{s}:=T_{s}(\Om)$ be defined as in Section \ref{section1a}.
Let $h$ be a given function in $H^1_0(\hold)$ such that $h\in H^2(\omega)$ where $\omega$ is an open set satisfying $\partial\Om\subset \omega$. This implies that $\partial\Omega_{s}\subset\omega$ for all $s\in [0,s_1]$ as long as $s_{1}$ is sufficiently small.
We consider the problem of finding the solution $u_s\in H^1_0(\Om_s)$ of
\begin{equation}\label{PDE_distribution}
\iomt\nabla u_s \cdot\nabla v =  \iomt \nabla h\cdot \nabla v + hv, \text{ for all }v\in H^1_0(\Om_s).
\end{equation} 
Note that the right-hand side of \eqref{PDE_distribution} defines a distribution in $H^{-1}(\Om_s)$.
The fact that  $\nabla h$ is only in $L^2(\hold)$ precludes the direct application of Theorem \ref{thm1}.
Noticing that the restriction of $h\in H^{1}_0(\hold)$ to $\Om_s$ belongs to $H^{1}(\Om_s)$, Green's formula yields
\begin{equation}\label{eq:160}
\iomt - u_s \Delta v = \iomt  -h \Delta v + hv + \int_{\partial\Om_s} h\dn v, \text{ for all }v\in H^2(\Om_s)\cap H^1_0(\Om_s).
\end{equation} 
Proceeding in a similar way as in Section \ref{sec:3.1}, we introduce, using the pullback $\Psi_s(\psi) = \psi\circ T_s$ in \eqref{eq:160}, the  function $\A:[0,s_1]\times E \rightarrow F^{*}$ defined as 
\begin{align*}
\langle  \A(s,\vp),\psi\rangle_{F^*,F} 
&:= \iomt - (\vp\circ\Tt^{-1}) \Delta (\psi\circ\Tt^{-1}) - \iomt  -h \Delta (\psi\circ\Tt^{-1}) + h(\psi\circ\Tt^{-1}) \\
&\quad - \int_{\partial\Om_s} h\nabla (\psi\circ\Tt^{-1})\cdot n_s  ,
\end{align*}
where $n_s$ denotes the outward unit normal vector to $\Om_s$.
Using the change of variables $x\mapsto \Tt(x)$  we get
\begin{align*}
\langle  \A(s,\vp),\psi\rangle_{F^*,F} & := \iom - \vp \Delta (\psi\circ\Tt^{-1})\circ\Tt \xi(s) - \iom  [-h^s \Delta (\psi\circ\Tt^{-1})\circ\Tt + h^s \psi ]\xi(s)\\
&\quad -  \int_{\partial\Om} h^s [DT_{s}^{-\transp}\nabla \psi] \cdot n_s\circ\Tt   \xi_\Gamma(s).
\end{align*}
with $h^s: = h\circ\Tt$.
As  $\ta\in \C^2_c(\hold,\Rd)$ and $\psi\in H^2(\Om)$, Lemma \ref{lem01} and Lemma \ref{lemma_second_order_chain} imply that $\A\in \C^{1}([0,s_{1}]\times E,F^{*})$.

Next, we define the operator  $A: E\to F^*$ as
\begin{align*}
\langle A\hat{\vp},\psi\rangle_{F^*,F}
& := \langle \partial_\vp\A(0,u)(\hat\vp ),\psi \rangle_{F^*,F} =   \iom - \hvp \Delta \psi .
\end{align*}
It is known that the adjoint operator $A^*: F = H^2(\Om)\cap H^1_0(\Om)\to E^* = L^2(\Om)$ is an isomorphism, see for instance \cite{MR812624}, consequently the operator $A: E = L^2(\Om)\to F^*$ is also  an isomorphism.
Note that this implies the existence of a unique $u_{s}\in L^{2}(\Omega)$ satisfying Equation \eqref{eq:160} for all $v\in H^{2}(\Omega_{s})\cap H_{0}^{1}(\Omega_{s})$. It is clear that the unique solution of \eqref{PDE_distribution} satisfies \eqref{eq:160}, therefore we conclude that $u_{s}\in H_{0}^{1}(\Omega_{s})$.

We have that $\partial_s (n_s\circ\Tt)|_{s=0} = -D_\Gamma \ta^\transp n$; see \cite[Lemma 5.5, p. 99]{Walker_book2015}.
Then, using \eqref{der_lap}, the fact that $\nabla \psi = (\dn\psi) n$ as $\psi$ is equal to zero along $\partial\Om$ - recall that $\psi\in F=H^2(\Om)\cap H^1_0(\Om)$ - and $\xi'_\Gamma(0) = \divv_\Gamma(\ta)$ (see Lemma \ref{lem01}), we compute
\begin{align*}
\langle L(u),\psi\rangle_{F^*,F} & := \langle  \partial_s \A(0,u),\psi\rangle_{F^*,F} \\
& = \iom (h- u) (-2 D^2\psi : D\theta - (\Delta \theta) \cdot \nabla\psi)
- u \Delta \psi \divv \theta
  - \iom  (-\Delta\psi + \psi)\divv(h\theta)\\
&\quad  -\int_{\po} \dn\psi \nabla h\cdot\ta + h \dn\psi (\divv_\Gamma(\ta) - n\cdot D\ta n - \underbrace{n\cdot D_\Gamma\ta n}_{=0}).
\end{align*}
Considering the regularity of $u,\psi,\ta,$ we have indeed $L(u)\in F^*$.

Now we consider the  cost functional $J:\mathds{P}(\mathcal{D})\to\R$ given by
\begin{equation}\label{cost_distribution}
 J(\Om) = \int_{\Om} \F(x,u), 
\end{equation}
where $\F\in\C^{1}(\mathcal{D}\times\R)$. We assume that there are positive constants $c_0$ and $c_1$ such that: 
\begin{align}
\label{F1.1}\max\left\{ \sup_{x\in\mathcal{D}}\left|\mathcal{F}(x,r)\right|,\sup_{x\in\mathcal{D}}\left|\nabla_{x}\mathcal{F}(x,r)\right|\right\}  & \le  c_0 + c_1 r^{2},\\
\label{F1.2} \sup_{x\in\mathcal{D}}\left|\partial_r \mathcal{F}(x,r)\right| & \le  c_0 + c_1 r.
\end{align}
\begin{example}
Let $u_{d}\in \C^{1}(\overline{\mathcal{D}})$. It is easy to see that $\mathcal{F}(x,r)=r-u_{d}(x)$ and $\mathcal{F}(x,r)=\left(r-u_{d}(x)\right)^{2}$ satisfy conditions \eqref{F1.1}-\eqref{F1.2}. If  $\alpha\in\left[1,2\right]$, we can also define  $\mathcal{F}(x,r) :=  |r-u_{d}(x)|^{1+\alpha}$. In this case, its derivatives 
\begin{align*}
\partial_r \mathcal{F}(x,r) & =  \left(1+\alpha\right)\text{sgn}(r-u_{d}(x))|r-u_{d}(x)|^{\alpha},\quad
\nabla_{x}\mathcal{F}(x,t) =  \left(1+\alpha\right)\text{sgn}(r-u_{d}(x))\left|r-u_{d}(x)\right|^{\alpha}\nabla u_{d}(x),
\end{align*}
satisfy \eqref{F1.1}-\eqref{F1.2}. Above, $sgn$ denotes the sign function.
\end{example}

We introduce the function 
$$ \B(s,\vp) := \iomt \F(x,\vp\circ\Tt^{-1}) 
= \iom \F(\Tt(x),\vp)\xi(s). $$
The proof of the following result is given in the Appendix.
\begin{proposition}\label{thm2}
The function $\B$ belongs to $\C^1([0,s_1]\times E , \R)$.
\end{proposition}
Using Proposition \ref{thm2} and Proposition \ref{prop:dphi_b} (see the Appendix)  we get
$$ \langle \bl(u),\hvp\rangle_{E^*,E} := \partial_\vp \B(0,u)(\hvp) 
= \iom \partial_r \F(x,u)\hvp. $$
Applying Theorem \ref{thm1} we obtain that the material derivative $\dot u$  is the unique solution of 
$$\langle A\dot u,\psi\rangle_{F^*,F} = -\langle L(u),\psi\rangle_{F^*,F}\quad \text{ for all }\psi\in F ,$$
i.e. $\dot u\in E = L^2(\Om)$ is solution of
\begin{align*}
\iom - \dot u\Delta \psi & = 
-\iom (h- u) (-2 D^2\psi : D\theta - (\Delta \theta) \cdot \nabla\psi)
- u \Delta \psi \divv \theta
+ \iom
(-\Delta\psi + \psi)\divv(h\theta)\\
&\quad + \int_{\po} \dn\psi \nabla h\cdot\ta + h \dn\psi (\divv_\Gamma(\ta) - n\cdot D\ta n )\quad
\text{ for all }\psi\in F.
\end{align*}
The adjoint state $p\in F$ is the unique solution of 
$$\langle A^*p,\hvp\rangle_{E^*,E} = - \langle \bl(u),\hvp\rangle_{E^*,E}\quad\text{ for all }\hvp\in E.$$
This means, using $\langle A^*p,\hvp\rangle_{E^*,E} = \langle p,A\hvp\rangle_{F^*,F}$, that $p\in F$ is solution of 
\begin{align*}
\iom - \hvp\Delta p  & =   -\iom \partial_r \F(x,u)\hvp.
\end{align*}
We observe that $p\in F = H^2(\Om)\cap H^1_0(\Om)$ whereas $\dot u$ only has the regularity $L^2(\Om)$.

Finally, Theorem \ref{thm1} and Proposition \ref{eq:prop4} yield the shape derivative
\begin{align}\label{dJ_distrib}
\begin{split}
dJ(\Om)(\ta) &=\langle L(u),p \rangle_{F^*,F} + \partial_s \B(0,u)\\
& =  \iom (h- u) (-2 D^2 p : D\theta - (\Delta \theta) \cdot \nabla p)
- u \Delta  p \divv \theta
- (-\Delta p +  p)\divv(h\theta)\\
&\quad  - \int_{\po} \dn p \nabla h\cdot\ta + h \dn p (\divv_\Gamma(\ta) - n\cdot D\ta n)
   + \iom \nabla_x \F(x,u)\cdot\theta + \F(x,u)\divv(\theta).
\end{split}
\end{align} 
Thus we have obtained the following result.
\begin{proposition}\label{prop5}
Let $\Om\in\mathds{P}(\mathcal{D})$ be of class  $\C^{2}$, $\ta\in \C^2_c(\hold,\Rd)$, and
suppose that $h\in H^1_0(\hold)$ is such that $h\in H^2(\omega)$, where $\omega$ is an open set satisfying $\partial\Om\subset \omega$.
Then, the distributed shape derivative of the cost functional \eqref{cost_distribution} subject to the constraint \eqref{PDE_distribution} is given by
\begin{align}\label{2nd_order_tensor}
dJ(\Om)(\ta) &= \iom S_0\cdot \theta + S_1 : D\theta + S_2 \tp D^2\theta + \int_{\po} S_{0,\Gamma}\cdot\ta + S_{1,\Gamma}: D\ta,
\end{align}
where $S_0\in L^1(\Om, \R^{d})$, $S_1\in L^1(\Om, \R^{d\times d})$, $S_2\in L^1(\Om, \R^{d\times d\times d})$, and $S_{0,\Gamma}\in L^1(\po, \R^{d}),S_{1,\Gamma}\in L^1(\po, \R^{d\times d})$   are given by 
\begin{align*}
S_0 & =  \nabla_x \F(x,u) + (\Delta p -  p)\nabla h, \\
S_1 & =  2 (u-h)  D^2 p  + [ h(\Delta p -  p) - u\Delta p + \F(x,u)]  I_d,\\
S_2 & = (u-h)\nabla p \otimes I_d,\\
S_{0,\Gamma} &= -\dn p \nabla h,\\
S_{1,\Gamma} &=  - h\dn p (I_d - 2 n\otimes n). 
\end{align*}
\end{proposition}
\begin{proof}
In view of \eqref{eq:55} and Lemma \ref{lemma_tensor} we have 
$$(u-h)(\Delta \ta)\cdot \nabla p
= (u-h) \tr(D^2\theta^{\transp}\nabla p )  
= (u-h) D^2\theta^{\transp}\nabla p : I_d
=  D^2\ta  \tp ((u-h)\nabla p\otimes I_d).$$
In this way we have identified the third-order tensor $S_2$.
The tensors $S_0,S_1$ are easily calculated using Lemma~\ref{lemma_tensor} and \eqref{dJ_distrib}.
The regularity of $S_0,S_1,S_2$ and $S_{0,\Gamma},S_{1,\Gamma}$ is an immediate consequence of the regularity of $\F$, $p\in H^2(\Om)\cap H^1_0(\Om)$,  $u\in H^1_0(\Om)$, $h \in H^2(\om)$, and $\ta\in \C^2_c(\hold,\Rd)$.

Regarding $S_{1,\Gamma}$, we observe that
$$ h \dn p (\divv_\Gamma(\ta) - n\cdot D\ta n) = h \dn p (\divv(\ta) - 2n\cdot D\ta n) = h \dn p D\ta : (I_d - 2 n\otimes n),$$
which yields the result.
\end{proof}

\subsection{Cost functional involving a second-order derivative}\label{sec:higher-order}
Let $\Om\in\mathds{P}(\mathcal{D})$ be a domain of class  $\C^{2}$, $\ta\in \C^2_c(\hold,\Rd)$,  the associated flow
$\Tt:\hold\rightarrow \hold$ and  $\Om_{s}:=T_{s}(\Om)$ be defined as in Section \ref{section1a}.
The example of this subsection also features $E\neq F$ and can be seen as the converse of the example of Section \ref{sec:distribution2}, in the sense that the adjoint is only in $L^2$ while the material derivative has $H^2$-regularity.
Using the fact that $\Om$ is of class  $\C^{2}$, a solution $u\in H^{2}(\Om)\cap H_{0}^{1}(\Om)$ of Poisson's equation for $f\in H^{1}(\mathcal{D})$ satisfies, using Green's formula and the density of $H_{0}^{1}(\Om)$ in $L^{2}(\Om)$, the following:
\begin{equation}\label{eq:555}
\int_\Om -v \Delta u=\int_\Om fv,\quad \text{for all}\,v\in L^{2}(\Om).
\end{equation}
This motivates the definition of $F = F^*= L^2(\Om)$ and $E = H^2(\Om)\cap H^1_0(\Om)$. 
Proceeding  as in Section \ref{sec:3.1}, we consider the same problem on $\Om_{s}$ and using a change of variables $x\mapsto T_s(x)$ we are led to define  $\A:[0,s_{1}]\times E \rightarrow F^{*}$ as
$$\langle  \A(s,\vp),\psi\rangle_{F^*,F}
=  \int_{\Om_s} -\psi\circ \Tt^{-1} \Delta(\vp\circ \Tt^{-1}) -f\psi\circ \Tt^{-1} 
= \iom -\psi \Delta(\vp\circ \Tt^{-1})\circ\Tt \xi(s) - f\circ \Tt \psi \xi(s).$$
Using Lemma \ref{lem01} and Lemma \ref{lemma_second_order_chain} (see the Appendix) we  immediately obtain $\A\in \C^{1}([0,s_{1}]\times E,F^{*})$.
Let $u^s\in E$ be the solution of 
\begin{equation}\label{115}
\A(s,u^s) = 0.
\end{equation}
The operator $A := \partial_\vp \A(0,u)$, $u=u^0$, is defined by
\begin{align*}
A: E & \to F^*,\\
\vp & \mapsto \left(F\ni \psi\mapsto \iom - \psi\Delta \vp\right)  .
\end{align*}
Note that $A: E\to F^*$ is an isomorphism since  $\Om$ is of class $\C^2$. 

We also compute, using \eqref{der_lap} of Lemma \ref{lemma_second_order_chain},
\begin{align}\label{eq:340}
\langle L(u),\psi\rangle_{F^*,F} := \langle  \partial_s \A(0,u),\psi\rangle_{F^*,F}
= \iom \psi ( 2D^2u : D\ta + (\Delta \ta)\cdot \nabla u) -\psi \Delta u \divv\ta
-\psi\divv(f\ta). 
\end{align}

Consider now the cost functional $J:\mathds{P}(\mathcal{D})\to \R$ given by
\begin{equation}\label{cost_D2u}
 J(\Om) := \frac{1}{2}\int_{\Om} |D^2 u|^2,
\end{equation}
and the corresponding perturbed functional
$$ \B(s,\vp) := \frac{1}{2}\int_{\Om_s} | D^2 (\vp\circ \Tt^{-1})|^2
= \frac{1}{2}\int_\Om | D^2 (\vp\circ \Tt^{-1})|^2\circ\Tt \xi(s).$$
Note that $\B(s,u^s)$ is well-defined for all $s\in [0,s_{1}]$ since $u^s\in E$.
Using $\vp\in E$ and \eqref{der_D2} we obtain
\begin{align}\label{eq:341}
 \partial_s \B(0,u) 
&= \int_\Om 
[- D\theta^{\transp} D^2 u-   D^2 u D\theta 
- D^2\theta^{\transp}\nabla u] : D^2 u 
+ \frac{1}{2}|D^2 u|^2 \divv(\ta).
\end{align}
We also compute
$$\langle \bl(u),\hvp\rangle_{E^*,E} := \langle\partial_\vp \B(0,u),\hvp\rangle_{E^*,E} = \int_\Om D^2 u : D^2\hvp.$$
The adjoint of $A$ is given by
\begin{align*}
A^*: F & \to E^*,\\
p & \mapsto \left(E\ni \vp\mapsto \iom - p\Delta \vp\right)  ,
\end{align*}
and the adjoint equation is
$$\langle A^*p,\hvp\rangle_{E^*,E} = - \langle \bl(u),\hvp\rangle_{E^*,E}\qquad\text{ for all }\hvp\in E,$$
which may be written as
$$ \iom - p\Delta \hvp = - \int_\Om D^2 u : D^2\hvp\qquad \text{ for all }\hvp\in E,$$
and this yields the regularity $p\in F = L^2(\Om)$.

The equation for the material derivative is
$$\langle A \dot{u},\psi\rangle_{F^*,F} = - \langle L(u),\psi\rangle_{F^*,F}\qquad \text{ for all }\psi\in F$$
which corresponds to
$$ \iom - \psi\Delta \dot{u} = - \iom \psi ( 2D^2u : D\ta + (\Delta \ta)\cdot \nabla u) -\psi \Delta u \divv\ta-\psi\divv(f\ta)\qquad \text{ for all }\psi\in F,$$
which yields the regularity $\dot{u}\in E =  H^2(\Om)\cap H^1_0(\Om)$.

Finally, Theorem \ref{thm1}, \eqref{eq:340} and \eqref{eq:341} yield the following expression for the shape derivative:
\begin{align}\label{dJ_higher_order}
\begin{split}
dJ(\Om)(\ta) &=\langle L(u),p \rangle_{F^*,F} + \partial_s \B(0,u)\\
& =  \iom p( 2D^2u : D\ta + (\Delta \ta)\cdot \nabla u) 
\underbrace{-p\Delta u \divv\ta -pf\divv(\ta)}_{=0 \text{ due to }-\Delta u =f}
 - p\nabla f\cdot\ta\\
& \quad + \int_\Om 
[- D\theta^{\transp} D^2 u-   D^2 u D\theta 
- D^2\theta^{\transp}\nabla u] : D^2 u 
+ \frac{1}{2}|D^2 u|^2 \divv(\ta).
\end{split}
\end{align} 
Using Lemma \ref{lemma_tensor} and $D^2 u^\transp = D^2 u$ we compute 
$$ D\theta^{\transp} D^2 u : D^2 u 
= (D\theta^{\transp} D^2 u)^\transp : (D^2 u)^\transp 
= D^2 u D\theta : D^2 u  = D\theta : (D^2 u)^2,$$
and $(D^2\theta^{\transp}\nabla u ): D^2 u 
= D^2\theta \tp (\nabla u \otimes D^2 u)$.
In view of \eqref{eq:55} and Lemma \ref{lemma_tensor} we also have 
$$p(\Delta \ta)\cdot \nabla u 
= p \tr(D^2\theta^{\transp}\nabla u )  
= p D^2\theta^{\transp}\nabla u : I_d
=  D^2\ta  \tp (p\nabla u\otimes I_d).$$
Using the regularity  $p\in L^2(\Om)$ and $u\in H^2(\Om)\cap H^1_0(\Om)$,  we obtain the following result.
\begin{proposition}\label{prop6}
Let $\Om\in\mathds{P}(\mathcal{D})$ be a domain of class  $\C^{2}$, $\ta\in \C^2_c(\hold,\Rd)$ and $f\in H^{1}(\mathcal{D})$.
Then, the distributed shape derivative of the cost functional \eqref{cost_D2u} subject to the constraint \eqref{eq:555} is
\begin{align}\label{2nd_order_tensor_b}
dJ(\Om)(\ta) &= \iom S_0\cdot \theta + S_1 : D\theta + S_2 \tp D^2\theta,
\end{align}
where $S_2\in L^1(\Om, \R^{d\times d\times d})$, $S_1\in L^1(\Om, \R^{d\times d})$, $S_0\in L^1(\Om, \R^{d})$ are given by $S_0 =  -p\nabla f$,
\begin{align*}
S_1 & =  2 p D^2 u -2 (D^2 u)^2  + \frac{1}{2}|D^2 u|^2 I_d,\ \text{ and }\ 
S_2  = -\nabla u \otimes D^2 u + p\nabla u\otimes I_d.
\end{align*}
\end{proposition}
\begin{remark}
Propositions \ref{prop5} and \ref{prop6} show that the second-order terms $D^2\ta$ appear in the shape derivative $dJ(\Om)(\ta)$  due to the presence of second-order derivatives in either $\A$ or $\mathcal{B}$.
\end{remark}

\section{Linear parabolic equations of second order}\label{sec:time}
Let $\Om\in\mathds{P}(\mathcal{D})$ be  of class  $\C^{2}$, $\ta\in \C^2_c(\hold,\Rd)$, the associated flow
$\Tt:\hold\rightarrow \hold$ and  $\Om_{s}:=T_{s}(\Om)$ be defined as in Section \ref{section1a}.
We assume that $M\in\C^{1}([0,t_{0}]\times\mathcal{D},\R^{d\times d})$
is a symmetric and uniformly positive definite matrix, that is $M(t,y)\ge CI_{d}>0$, for all $(t,y)\in[0,t_0]\times\mathcal{D}$ and some constant
$C>0$.
Suppose also that $f\in L^{2}(0,t_{0};H^{1}(\mathcal{D}))$ and $g\in H^1(\mathcal{D})$.
The dual pairing between
$H^{-1}(\Om_{s})$ and $H_{0}^{1}(\Om_{s})$ is denoted by $\left\langle \cdot , \cdot \right\rangle _{H^{-1}(\Om_{s}),H_{0}^{1}(\Om_{s})}$.

In this section we employ the following spaces:
\begin{align*}
E: & =L^{2}(0,t_{0};H_{0}^{1}(\Om))\cap H^{1}(0,t_{0};H^{-1}(\Om)),\\
F: & =L^{2}(0,t_{0};H_{0}^{1}(\Om))\oplus L^{2}(\Om),\\
F^{*}: & =L^{2}(0,t_{0};H^{-1}(\Om))\oplus L^{2}(\Om).
\end{align*}
We recall that $E$ is a subset of $\C([0,t_0],L^2(\Om))$, see \cite[Section 3, Chapter 1]{lionsmagenes}. The dual pair $\left\langle \cdot ,\cdot \right\rangle _{F^{*},F}$ is defined as
\[
\left\langle (p,q),(\psi_1,\psi_2)\right\rangle _{F^{*},F}=\int_{0}^{t_{0}}\left\langle p,\psi_1\right\rangle _{H^{-1}(\Om),H_{0}^{1}(\Om)}dt+\int_{\Om}q\psi_2 dx.
\]

We consider the following linear parabolic equation:
\begin{align}\label{eq:parabolicmain}
\begin{split}
\partial_{t} u_s -\divv\left(M\nabla u_s\right) & =  f \quad  \textrm{ in } (0,t_{0})\times \Om_{s},\\
u_s & =  0   \quad \textrm{ on }(0,t_{0})\times\partial\Om_{s},\\
u_s & =  g \quad   \textrm{ in }\{0\}\times \Om_{s}.
\end{split}
\end{align}
A weak solution of Equation \eqref{eq:parabolicmain} is a function
$u_{s}\in L^{2}(0,t_{0};H_{0}^{1}(\Om_{s}))\cap H^{1}(0,t_{0};H^{-1}(\Om_{s}))$ satisfying
\begin{align}\label{eq:u_s_omega_s}
\begin{split}
\left\langle \partial_{t} u_{s}(t),\psi\right\rangle _{H^{-1}(\Om_{s}),H_{0}^{1}(\Om_{s})}+\int_{\Om_{s}}M(t)\nabla u_s \cdot \nabla\psi & =  \left\langle f(t),\psi\right\rangle _{H^{-1}(\Om_{s}),H_{0}^{1}(\Om_{s})} 
,\\
u_{s} & =  g \quad   \textrm{ in }\{0\}\times \Om_{s},
\end{split}
\end{align}
for all $\psi\in H_{0}^{1}(\Om_{s})$ and  a.e $t\in[0,t_0]$. 
Recall that if $u_{s}$ is a weak solution of \eqref{eq:parabolicmain}, then $u_{s}\in \C([0,t_{0}],L^{2}(\Om_{s}))$
and $u_{s}$ is well-defined at $t=0$.
The assumptions on $M$ ensure the existence of a unique weak solution $u_{s}$ of Equation \eqref{eq:parabolicmain}; see \cite[Theorem 4.1 of Chapter 3]{lionsmagenes}.

As in the previous subsections, we define the function $\A$ via a pullback. 
First of all one needs to recall the definition of pullback for distributions; see \cite[Definition 3.19]{MR2453959}.
\begin{definition}
The pullback of distributions $\Psi_s : H^{-1}(\Om_s)\to H^{-1}(\Om)$ is defined by
$$ \left\langle \Psi_s(h),\psi\right\rangle _{H^{-1}(\Om),H_{0}^{1}(\Om)}
=
\left\langle h, \xi(s)^{-1}\psi\circ T_s^{-1}\right\rangle _{H^{-1}(\Om_{s}),H_{0}^{1}(\Om_{s})}
\quad \forall \psi\in H^1_0(\Om),$$
with $h\in H^{-1}(\Om_s)$.
The associated superposition operator $\overline{\Psi}_s$ is defined by
\begin{align}\label{eq:psi_us}
\begin{split}
\overline{\Psi}_s: H^{1}(0,t_{0};H^{-1}(\Om_s)) & \to H^{1}(0,t_{0};H^{-1}(\Om)),\\
v & \mapsto [  t\mapsto (\overline{\Psi}_s v)(t)  := \Psi_s(v(t))    ].
\end{split}
\end{align}
\end{definition}
\begin{remark} When restricted to $L^{2}(\Om_s)$, $\Psi_{s}$ corresponds to the usual pullback $\Psi_{s}(h)=h\circ T_{s}$ if $h\in L^{2}(\Om_s)$. 
Notice that the superposition operator $\overline{\Psi}_{s}$ can also be defined as in \eqref{eq:psi_us} for other classes of functions such as  $\overline{\Psi}_{s}: L^{2}(0,t_{0};H^{1}(\Omega_{s}))\to  L^{2}(0,t_{0};H^{1}(\Om))$.
\end{remark}

Now define the transported solution  $u^s := \overline{\Psi}_s u_s\in L^{2}(0,t_{0};H_{0}^{1}(\Om))\cap H^{1}(0,t_{0};H^{-1}(\Om))$.
It can be shown that the superposition operator $\overline{\Psi}_s$ and the partial derivative $\partial_t$ commute, i.e. we have
$$\partial_{t} u^s(t,x) 
= \partial_{t} [( \overline{\Psi}_s u_{s})(t,x)]
= \overline{\Psi}_s \left[ \partial_{t}  u_{s}\right](t,x).$$
Introduce the notations  $f^{s}:=\overline{\Psi}_s f$ and $g^{s}:= \Psi_s  g$.
Using a precomposition with $\Psi_s$ in  \eqref{eq:u_s_omega_s} and using \eqref{eq:psi_us}, we obtain that $u^{s}$ is the unique function satisfying, for a.e. $t\in[0,t_0]$,
\begin{align*}
\left\langle \xi(s)\partial_{t} u^{s}(t),\psi\right\rangle _{H^{-1}(\Om),H_{0}^{1}(\Om)}
+\int_{\Om} \M(s,M_s(t)) \nabla u^s(t)\cdot  \nabla\psi & = \left\langle \xi(s)f^{s}(t),\psi\right\rangle _{H^{-1}(\Om),H_{0}^{1}(\Om)},\\
u^{s} & = g^{s} \quad   \textrm{ in }\{0\}\times \Om ,
\end{align*}
for all $\psi\in H_{0}^{1}(\Om)$, where we have used Definition \ref{def5} and the notation $M_s(t):\hold\to\R^{d\times d}$ given by
\begin{equation}\label{eq:mMs}
M_s(t)(x):=M(t,T_{s}(x)). 
\end{equation}

The assumptions on the matrix $M$ and the Lax-Milgram theorem ensure
the existence of an isomorphism $\mathcal{H}_{s}(t):H_{0}^{1}(\Om)\to H^{-1}(\Om)$,
$s\in[0,s_{1}]$ and $t\in[0,t_{0}]$, satisfying
\[
\left\langle \mathcal{H}_{s}(t)\vp,\psi\right\rangle _{H^{-1}(\Om),H_{0}^{1}(\Om)}
= \int_{\Om} \M(s,M_s(t)) \nabla\vp\cdot  \nabla\psi.
\]
For each $s\in[0,s_{1}]$, we denote by $\mathcal{A}_0(s):E\to F^{*}$
the operator given by: 
\begin{equation}\label{eq:B(s)}
(\mathcal{A}_0(s)\vp)(t) :=(\begin{array}{cc}
\xi(s)\partial_{t}\vp(t)+\mathcal{H}_{s}(t)\vp(t), & \vp(0)\xi(s)\end{array}).
\end{equation}
Let us define the operator $\mathcal{A}:[0,s_{1}]\times E\to F^{*}$ by 
\begin{equation}\label{eq:ABE}
\mathcal{A}(s,\vp) :=\mathcal{A}_0(s)\vp-\mathcal{A}_1(s), 
\end{equation}
where  $\mathcal{A}_1(s) :=(\xi(s)f^{s} ,\xi(s)g^{s} )\in F^{*}$ is given
by 
\[
\mathcal{A}_1(s)(\psi_1,\psi_2)  :=\int_{0}^{t_{0}} \iom  \xi(s)f^{s} \psi_1\, +\int_{\Om}  \xi(s) g^{s}\psi_2,
\quad \forall (\psi_1,\psi_2)\in F.
\]
Note that $\mathcal{A}(s,v)=0$ if and only if $v=u^{s}$, 
that is, the solutions of $\mathcal{A}(s,v)=0$ are precisely the
weak solutions of the parabolic equation \eqref{eq:parabolicmain} after
the precomposition $u^s := \overline{\Psi}_s u_s$.

Gathering the above informations, we obtain the following result.
\begin{theorem}
\label{thm:A_parabolic} Under the above assumptions,
the function $\mathcal{A}$ belongs to $\mathcal{C}^{1}([0,s_{1}]\times E,F^{*})$
and $A:=\partial_{\varphi}\mathcal{A}(0,u):E\to F^{*}$ is equal to
$\mathcal{A}_{0}(0)$ and is an isomorphism.
\end{theorem}

\begin{proof}
Observe that for $\psi=(\psi_{1},\psi_{2})\in F$, we have
\begin{equation}
\begin{aligned}\left\langle \mathcal{A}(s,\varphi),\psi\right\rangle _{F^{*},F} & =\left\langle \mathcal{A}_{0}(s)\varphi,\psi\right\rangle _{F^{*},F}-\left\langle \mathcal{A}_{1}(s),\psi\right\rangle _{F^{*},F}\\
 & =\int_{0}^{t_{0}}\left\langle \xi(s)\partial_{t}\varphi(t),\psi_{1}\right\rangle _{H^{-1}(\Omega),H_{0}^{1}(\Omega)}+\int_{0}^{t_{0}}\int_{\Omega}\mathcal{M}(s,M_{s}(t))\nabla\varphi(t)\cdot\nabla\psi_{1}\\
 & \quad-\int_{0}^{t_{0}}\int_{\Omega}\xi(s)f^{s}(t)\psi_{1}+\int_{\Omega}\varphi(0)\psi_{2}\xi(s)-\int_{\Omega}g^{s}\psi_{2}\xi(s).
\end{aligned}
\label{eq:Asvarphiparabolic}
\end{equation}
Using Lemma \ref{lem01}, the assumed regularity $f\in L^{2}(0,t_{0};H^{1}(\mathcal{D}))$
and $g\in H^{1}(\mathcal{D})$, we can formally derive the terms inside
the integrals of Equation \eqref{eq:Asvarphiparabolic} to find $\partial_{\varphi}\mathcal{A}(s,\varphi)$
and $\partial_{s}\mathcal{A}(s,\varphi)$. It is not difficult to
see that the formal computation leads to the right Fréchet derivatives
and that these derivatives are continuous.

However, the term $\int_{0}^{t_{0}}\left\langle \xi(s)\partial_{t}\varphi(t),\psi_{1}\right\rangle _{H^{-1}(\Omega),H_{0}^{1}(\Omega)}$
requires a little more attention, since we are dealing with distributions.
Formally we have that
\[
\partial_{s}\int_{0}^{t_{0}}\left\langle \xi(s)\partial_{t}\varphi(t),\psi_{1}\right\rangle _{H^{-1}(\Omega),H_{0}^{1}(\Omega)}=\int_{0}^{t_{0}}\left\langle \xi'(s)\partial_{t}\varphi(t),\psi_{1}\right\rangle _{H^{-1}(\Omega),H_{0}^{1}(\Omega)}.
\]
Using the notation $F_{1}:=L^{2}(0,t_{0};H^{1}(\Omega))$ and $F_{1}^{*}:=L^{2}(0,t_{0};H^{-1}(\Omega))$,
we need to show that
\[
\lim_{h\to0}\left\Vert z(h)\partial_{t}\varphi\right\Vert _{F_{1}^{*}}=0\quad
\text{ where }\quad
z(h):=\frac{\xi(s+h)-\xi(s)}{h}-\xi'(s).
\]
In view of Lemma \ref{lem01} we have $\xi'\in \C^0([0,s_0],\C^1(\mathcal{D}))$, therefore
\begin{align*}
\left|\int_{0}^{t_{0}}\left\langle z(h)\partial_{t}\varphi,\psi_{1}\right\rangle _{H^{-1}(\Omega),H_{0}^{1}(\Omega)}\right| & \leq\int_{0}^{t_{0}}\|\partial_{t}\varphi\|_{H^{-1}(\Omega)}\|z(h)\psi_{1}\|_{H_{0}^{1}(\Omega)}\\
 & \leq\left(\int_{0}^{t_{0}}\|\partial_{t}\varphi\|_{H^{-1}(\Omega)}^{2}\right)^{1/2}\left(\int_{0}^{t_{0}}\|z(h)\psi_{1}\|_{H_{0}^{1}(\Omega)}^{2}\right)^{1/2}\\
 & \leq\|\partial_{t}\varphi\|_{F_{1}^{*}}\|z(h)\|_{\mathcal{C}^{1}}\|\psi_{1}\|_{F_{1}}.
\end{align*}
Hence it follows that
\begin{align*}
\sup_{\|\psi_{1}\|_{F_{1}}\leq1}\left|\int_{0}^{t_{0}}\left\langle z(h)\partial_{t}\varphi,\psi_{1}\right\rangle _{H^{-1}(\Omega),H_{0}^{1}(\Omega)}\right| & \leq\|\partial_{t}\varphi\|_{F_{1}^{*}}\|z(h)\|_{\mathcal{C}^{1}}.
\end{align*}
Again, as $\xi'\in \C^0([0,s_0],\C^1(\mathcal{D}))$ we get $\|z(h)\|_{\mathcal{C}^{1}}\to0$
as $h\to0$; 
this proves that $\left\Vert z(h)\partial_{t}\varphi\right\Vert _{F_{1}^{*}}\to0$.

Finally, in view of \eqref{eq:ABE} and the linearity
of $\mathcal{A}_{0}(s)$ in $\varphi$, we see that $A:=\partial_{\varphi}\mathcal{A}(0,u)=\mathcal{A}_{0}(0)$,
where
\[
\mathcal{A}_{0}(0)\varphi(t):=(\begin{array}{cc}
\partial_{t}\varphi(t)+\mathcal{H}_{0}(t)\varphi(t), & \varphi(0)\end{array}).
\]
Consequently,  the operator $A:E\to F^{*}$ is an isomorphism by the maximal regularity of parabolic equations; see for
instance \cite[Chapter 3, Theorem 4.1]{lionsmagenes}.
\end{proof}

\subsection{Adjoint state equation}
In this section we obtain the equation for the adjoint state, for two different cost functionals.
Unlike the problems studied in the other sections, the adjoint equation in the parabolic case requires an integration by part in time to become more explicit.
In this section we still assume that $f\in L^{2}(0,t_{0};H^{1}(\mathcal{D}))$ and $g\in H^1(\mathcal{D})$.
Recall that $\Om\in \mathds{P}(\mathcal{D})$ is of class $\C^2$ and $u^{s}\in E$ is the solution of $\mathcal{A}(s,u^{s})=0$.
\subsubsection{Adjoint for the first cost function}\label{sec:adj_J1}
Let $u_{d}\in L^{2}(0,t_{0};H^1(\mathcal{D}))$ and consider the cost functional $J_1:\mathds{P}(\mathcal{D})\to\R$ given by
\begin{equation}\label{cost1_parabolic}
J_{1}(\Om) :=\frac{1}{2}\int_{0}^{t_{0}}\iom(u_{s}-u_{d})^{2}.
\end{equation}
Using the change of variables $\Om \ni x\mapsto T_{s}(x)\in\Om_{s}$ we obtain
\[
J_1(\Om_{s})=\frac{1}{2}\int_{0}^{t_{0}}\int_{\Om_{s}}(u_{s}-u_{d})^{2}=\frac{1}{2}\int_{0}^{t_{0}}\int_{\Om}(u^{s}-u_{d}\circ T_{s})^{2}\xi(s).
\]
Thus,
we define the function $\B\in \C^{1}([0,s_{1}]\times E,\R)$
by
\begin{equation}\label{eq:bsvp-2}
\B(s,\vp):=\frac{1}{2}\int_{0}^{t_{0}}\int_{\Om}(\vp-u_{d}\circ T_{s})^{2}\xi(s)
\end{equation}
 and we compute 
\[
\left\langle \bl(u),\hvp\right\rangle _{E^{*},E} := \langle\partial_{\vp}\B(0,u),\hvp\rangle_{E^{*},E}
=\int_{0}^{t_{0}}\int_{\Om}\hvp(u-u_{d}).
\]
Hence, the linear form $\bl(u)\in E^{*}$ can be identified with $(u-u_{d})|_{\Om}\in  L^{2}(0,t_{0};L^{2}(\Om))$.

In view of Theorem \ref{thm:A_parabolic},  $A:=\partial_{\vp}\mathcal{A}(0,u):E\to F^{*}$ is an isomorphism and
we can apply Theorem \ref{thm1}.
Hence, there exists
a unique solution $(p,q)\in F$ to the adjoint state equation 
\begin{align}\label{eq:adjoint_parab1}
\left\langle A^{*}(p,q),\vp\right\rangle _{E^{*},E}=-\left\langle \bl(u),\vp\right\rangle _{E^{*},E}\quad \text{ for all }\vp\in E.
\end{align}
We now seek an explicit expression of the adjoint equation. 
\begin{lemma}\label{lem:04}
The solution $(p,q)\in F$ of \eqref{eq:adjoint_parab1} satisfies $p\in E= L^{2}(0,t_{0};H_{0}^{1}(\Om))\cap H^{1}(0,t_{0};H^{-1}(\Om))$ and $q=p(0)$, and
 is the unique weak solution of the following backwards parabolic  equation with terminal condition:
\begin{align}\label{eq:adjJ1}
\begin{split}
-\partial_{t} p +\divv\left(M\nabla p\right) & =  -(u-u_d) \quad  \textrm{ in } (0,t_{0})\times \Om,\\
p & =  0   \quad \textrm{ on }(0,t_{0})\times\partial\Om,\\
p & =  0 \quad   \textrm{ in }\{t_0\}\times \Om.
\end{split}
\end{align}
\end{lemma}
\begin{proof}
We
have by definition of $A^{*}$ that 
\[
\left\langle A\vp,(p,q)\right\rangle _{F^{*},F}=-\left\langle \bl(u),\vp\right\rangle _{E^{*},E}\quad \text{ for all }\vp\in E.
\]
This is equivalent to, using the fact that $\varphi(0)\in L^2(\Om)$ due to $E\subset\C([0,t_0],L^2(\Om))$, 
\begin{equation}\label{eq:10}
\int_{0}^{t_{0}}\left\langle \partial_{t}\vp,p\right\rangle _{H^{-1}(\Om),H_{0}^{1}(\Om)}dt
+\int_{0}^{t_{0}}\int_{\Om}M\nabla\vp\cdot\nabla p
+\int_{\Om}\vp(0)q=-\int_{0}^{t_{0}}\int_{\Om}\vp(u-u_{d})\text{ for all }\vp\in E.
\end{equation}

Now let us consider $\tilde{p}\in E$ the unique weak solution of
\begin{align*}
\begin{split}\partial_{t}\tilde{p}(t)-\text{div}\left(M(t_{0}-t)\nabla\tilde{p}(t)\right) & =-(u(t_{0}-t)-u_{d}(t_{0}-t))\quad\textrm{ in }(0,t_{0})\times\Omega_{s},\\
\tilde{p} & =0\quad\textrm{ on }(0,t_{0})\times\partial\Omega_{s},\\
\tilde{p} & =0\quad\textrm{ in }\{0\}\times\Omega_{s}.
\end{split}
\end{align*}
By the properties of the matrix $M$ and as $(u(t_{0}-t)-u_{d}(t_{0}-t))\in F$, this solution exists and its unique \cite[Theorem 4.1 of Chapter 3]{lionsmagenes}. By definition, it satisfies: 
\begin{align*}
& \int_{0}^{t_{0}}\left\langle \partial_{t}\tilde{p},\vp\right\rangle _{H^{-1}(\Om),H_{0}^{1}(\Om)}  dt
+\int_{0}^{t_{0}}\int_{\Om}M(t_0 -t)\nabla\tilde{p}\cdot\nabla\vp 
+\int_{\Om}\tilde{p}(0)\hat{\psi}(0)\\
& \qquad=-\int_{0}^{t_{0}}\int_{\Om}\vp(u(t_{0}-t)-u_{d}(t_{0}-t))\quad\text{ for all }(\vp,\hat{\psi})\in F.
\end{align*}
Let us choose specific test functions $\vp\in E$
and $\hat{\psi}=\vp(0)\in L^{2}(\Om)$.
Integrating by part in time the term depending on  $\partial_{t}\tilde{p}$, also using the fact that $M$ is symmetric, we get 
\begin{align*}
& -\int_{0}^{t_{0}}\left\langle \partial_{t}\vp,\tilde{p}\right\rangle _{H^{-1}(\Om),H_{0}^{1}(\Om)}dt
+\int_{0}^{t_{0}}\int_{\Om}M(t_0-t)\nabla\vp\cdot\nabla \tilde{p}
+\int_{\Om}\tilde{p}(t_{0})\vp(t_{0})\\
& \qquad =-\int_{0}^{t_{0}}\int_{\Om}\vp(u(t_{0}-t)-u_{d}(t_{0}-t))\quad \text{ for all }\vp\in E.
\end{align*}
Now introduce $p^{\dagger}(t)=\tilde{p}(t_{0}-t)$ and $\vp^{\dagger}(t)=\vp(t_{0}-t)$,
we get using the change of variable $t\mapsto t_{0}-t$: 
\begin{align}\label{eq:12}
\begin{split}
& \int_{0}^{t_{0}}\left\langle \partial_{t}\vp^{\dagger},p^{\dagger}\right\rangle _{H^{-1}(\Om),H_{0}^{1}(\Om)}dt+\int_{0}^{t_{0}}\int_{\Om}M(t)\nabla\vp^{\dagger}\cdot\nabla p^{\dagger}+\int_{\Om}p^{\dagger}(0)\vp^{\dagger}(0)\\
& \qquad =-\int_{0}^{t_{0}}\int_{\Om}\vp^{\dagger}(u-u_{d})\quad \text{ for all }\vp^{\dagger}\in E.
\end{split}
\end{align}
In view of \eqref{eq:10} and the uniqueness of the solution $(p,q)\in F$ of \eqref{eq:adjoint_parab1},
this shows that $p^{\dagger}=p$ and $q=p^{\dagger}(0) =p(0)$.
Also, the adjoint $p$ has the higher regularity $p\in E$ since $p^{\dagger}\in E$, therefore we can integrate by part \eqref{eq:10} in time which yields
\begin{align}\label{eq:13-2}
\begin{split}
& -\int_{0}^{t_{0}}\left\langle \partial_{t}p,\vp\right\rangle _{H^{-1}(\Om),H_{0}^{1}(\Om)}dt+\int_{0}^{t_{0}}\int_{\Om}M\nabla\vp\cdot\nabla p+\int_{\Om}p(t_{0})\vp(t_{0})\\
& \qquad =-\int_{0}^{t_{0}}\int_{\Om}\vp(u-u_{d})\quad \text{ for all }\vp\in E.
\end{split}
\end{align}
Finally, using the fact that $M$ is symmetric we obtain \eqref{eq:adjJ1}.
\end{proof}
\subsubsection{Adjoint for the second cost function}
The second cost functional is $J_2:\mathds{P}(\mathcal{D})\to\R$ given by
\begin{equation}\label{cost2_parabolic}
J_{2}(\Om) :=\frac{1}{2}\iom (u(t_{0})-u_{d})^{2},
\end{equation}
where $u_{d}\in H^{1}(\mathcal{D})$ is a given function independent of time.
Proceeding in a similar way as in the previous subsections, using the change of variables $\Om\ni x\mapsto T_{s}(x)\in\Om_{s}$ in $J_{2}(\Om_s)$, we introduce the auxiliary function
\begin{equation}\label{eq:bsvp-3}
\B(s,\vp):=\frac{1}{2}\int_{\Om}(\vp(t_{0})-u_{d}\circ T_{s})^{2}\xi(s).
\end{equation}
We compute
\[
\langle \bl(u),\hvp\rangle_{E^{*},E} :=\partial_{\vp}\B(0,u)(\hvp) =\int_{\Om}\hvp(t_{0})(u(t_{0})-u_{d}).
\]
 Thus the linear form $\bl(u)\in E^{*}$ can be identified with $\delta_{t_{0}}\otimes(u(t_{0})-u_{d})|_{\Om}$,
where $\delta_{t_{0}}$ is a Dirac measure at time $t_{0}$. Here $\otimes$ stands for the tensor product of distributions.

Then, the adjoint state is the
unique solution $(p,q)\in F$ to the equation 
\begin{align}\label{eq:adjoint_parab1b}
\left\langle A^{*}(p,q),\vp\right\rangle _{E^{*},E}=-\left\langle \bl(u),\vp\right\rangle _{E^{*},E}\quad \text{ for all }\vp\in E.
\end{align}
We now seek an explicit expression of the adjoint equation. 
\begin{lemma}
The solution $(p,q)\in F$ of \eqref{eq:adjoint_parab1b} satisfies $p\in E= L^{2}(0,t_{0};H_{0}^{1}(\Om))\cap H^{1}(0,t_{0};H^{-1}(\Om))$ and $q=p(0)$, and
 is the unique weak solution of the following backwards parabolic  equation with terminal condition:
\begin{align}\label{eq:adjJ2}
\begin{split}
-\partial_{t} p +\divv\left(M\nabla p\right) & = 0 \quad  \textrm{ in } (0,t_{0})\times \Om,\\
p & =  0   \quad \textrm{ on }(0,t_{0})\times\partial\Om,\\
p & =   -(u(t_0) - u_d) \quad   \textrm{ in }\{t_0\}\times \Om.
\end{split}
\end{align}
\end{lemma}
\begin{proof}
The proof of this result is similar to the proof of Lemma \ref{lem:04}, thus
we only highlight here the main differences. 
We define $\tilde{p}\in E$
as the unique solution to
\begin{align*}
& \int_{0}^{t_{0}}\left\langle \partial_{t}\tilde{p},\vp\right\rangle _{H^{-1}(\Om),H_{0}^{1}(\Om)}  dt
+\int_{0}^{t_{0}}\int_{\Om}M(t_0 -t)\nabla\tilde{p}\cdot\nabla\vp 
+\int_{\Om}\tilde{p}(0)\hat{\psi}(0)\\
& \quad =  -\int_{\Om}\hat{\psi}(u(t_{0})-u_{d})\quad \text{ for all }(\vp,\hat{\psi})\in F.
\end{align*}
In particular we can choose specific test functions $\vp\in E$
and $\hat{\psi}=\vp(0)\in L^{2}(\Om)$.
Following the same steps as in the proof of Lemma \ref{lem:04},
this leads to 
\begin{align}\label{eq:14-2}
\begin{split}
& -\int_{0}^{t_{0}}\int_{\Om}\left\langle \partial_{t}p,\vp^{\dagger}\right\rangle _{H^{-1}(\Om)\times H_{0}^{1}(\Om)}+\int_{0}^{t_{0}}\int_{\Om}M\nabla\vp^{\dagger}\cdot\nabla p+\int_{\Om}p(t_{0})\vp^{\dagger}(t_{0})\\
&\qquad  = -\int_{\Om}\vp^{\dagger}(t_{0})(u(t_{0})-u_{d})\quad \text{ for all }\vp^{\dagger}\in E. 
\end{split}
\end{align}
This shows that $p\in E$
is the unique weak solution of \eqref{eq:adjJ2}.
\end{proof}
\subsubsection{Shape derivative of cost functionals}
First of all  we provide the equation for the material derivative. 
In view of \eqref{eq:B(s)}-\eqref{eq:ABE} we have
\begin{align}\label{eq:15}
\begin{split}
\langle  \A(s,\vp),\psi\rangle_{F^*,F} 
&=\int_0^{t_{0}}  \left\langle  \xi(s)   \partial_{t} \vp,\psi_1\right\rangle_{H^{-1}(\Om),H_{0}^{1}(\Om)}  
+ \int_0^{t_{0}} \iom   \M(s,M_s(t))\nabla \vp \cdot \nabla \psi_1  \\
&\quad + \iom \vp(0) \psi_2 \xi(s)
-\int_{0}^{t_{0}} \iom f^{s}\psi_1 \xi(s)
-\iom g^{s}  \psi_2 \xi(s),
\end{split}
\end{align}
where  $\vp\in E$,  $\psi =(\psi_1,\psi_2)\in F$ and $M_s$ is defined in \eqref{eq:mMs}.
In view of \eqref{eq:mMs} we have $\left. \partial_s M_s \right|_{s=0}= (DM) \ta$ where $DM\in\R^{d\times d\times d}$ is a third-order tensor.
Thus, using Theorem \ref{thm:A_parabolic} and Lemma \ref{lem01} we compute 
\begin{align}\label{eq:16}
\begin{split} 
\langle L(u),\psi\rangle_{F^*,F} 
& := \langle  \partial_s \A(0,u),\psi\rangle_{F^*,F} \\
& =\int_0^{t_{0}} \left\langle \partial_t u, \psi_1 \divv(\ta)\right\rangle_{H^{-1}(\Om),H_{0}^{1}(\Om)}
+ \int_0^{t_0}\iom (\M'(0,M) + (DM)\ta)\nabla u \cdot \nabla \psi_1  \\
& \quad + \iom u(0) \psi_2 \divv(\ta)
 - \int_0^{t_{0}} \iom \psi_1\nabla f\cdot \ta + \psi_1 f\divv(\ta) - \iom \psi_2 \nabla g\cdot \ta + \psi_2 g\divv(\ta). 
\end{split}
\end{align}
We have $\partial_t u \in L^{2}(0,T;H^{-1}(\Om))$, $\psi_1  \in L^{2}(0,T;H^{1}_{0}(\Om))$, $\nabla u \in L^{2}(0,T; L^{2}(\Om))$,  $\nabla \psi_1 \in L^{2}(0,T; L^{2}(\Om))$, $u(0)\in L^2(\Om)$ and $\psi_2  \in L^{2}(\Om)$; hence we get indeed $L(u)\in F^*$.
Then,  the equation for the material derivative $\dot{u}\in E$ is
$$\langle A \dot{u},\psi\rangle_{F^*,F} = - \langle L(u),\psi\rangle_{F^*,F}\qquad \text{ for all }\psi\in F,$$
which corresponds to
\begin{equation*}
\int_{0}^{t_{0}}\left\langle \partial_{t}\dot{u},\psi_1\right\rangle _{H^{-1}(\Om),H_{0}^{1}(\Om)}dt
+\int_{0}^{t_{0}}\int_{\Om}M\nabla\dot{u}\cdot\nabla \psi_1
+\int_{\Om}\dot{u}(0)\psi_2= - \langle L(u),\psi\rangle_{F^*,F}\text{ for all }\psi\in F.
\end{equation*}

The partial derivative of $\B$ given by \eqref{eq:bsvp-2} with respect to $s$ at $(0,u)$ exists thanks to Lemma \ref{lem01} and $u_{d}\in L^{2}(0,t_{0};H^1(\mathcal{D}))$.
Also, the partial derivative of $\B$ given by \eqref{eq:bsvp-3} with respect to $s$ at $(0,u)$ exists thanks to Lemma \ref{lem01} and $u_{d}\in H^1(\mathcal{D})$.
Consequently, using the result of Theorem \ref{thm:A_parabolic}, we can apply Theorem \ref{thm1} and we obtain
$$ dJ_1(\Om)(\ta) = \langle L(u),p \rangle_{F^*,F} + \partial_s \B(0,u),$$
which yields, in view of \eqref{eq:bsvp-2} and the fact that $p\in E$ as shown in Section \ref{sec:adj_J1},
\begin{align*} 
dJ_1(\Om)(\ta) & = \int_0^{t_{0}} \left\langle   \partial_{t} u, \divv(\ta)   p \right\rangle_{H^{-1}(\Om),H_{0}^{1}(\Om)}  
+\int_0^{t_{0}} \iom  (\M'(0,M) + (DM)\ta)\nabla u \cdot \nabla p  \\
&\quad + \iom u(0) p(0) \divv(\ta)  - \int_0^{t_{0}} \iom p\nabla f\cdot \ta + pf\divv(\ta) 
- \iom p(0)\nabla g\cdot \ta + p(0)g\divv(\ta) \\
&\quad +\int_0^{t_{0}} \iom -(u - u_d)\nabla u_d\cdot \ta + \frac{1}{2}(u - u_d)^2 \divv(\ta).
\end{align*}
Gathering these informations and using also that $u(0) =g$, we obtain the following result.
\begin{proposition}\label{prop:3}
Let $\Om\in\mathds{P}(\mathcal{D})$ be  of class  $\C^{2}$, $\ta\in \C^2_c(\hold,\Rd)$, $f\in L^{2}(0,t_{0};H^{1}(\mathcal{D}))$, $g\in H^1(\mathcal{D})$, $u_{d}\in L^{2}(0,t_{0};H^1(\mathcal{D}))$, and $M\in\C^{1}([0,t_{0}]\times\mathcal{D},\R^{d\times d})$ 
be symmetric and uniformly positive definite.
Then, the shape derivative of the cost functional \eqref{cost1_parabolic} subject to the constraint \eqref{eq:parabolicmain} is given by
\begin{align} \label{eq:sder_J1}
dJ_1(\Om)(\ta) & =  \int_0^{t_{0}} \left\langle   \partial_{t} u, \divv(\ta)   p \right\rangle_{H^{-1}(\Om),H_{0}^{1}(\Om)} +\iom S_0\cdot \ta + S_1 : D\ta ,
\end{align}
with $S_0\in L^1(\Om, \R^{d})$, $S_1\in L^1(\Om, \R^{d\times d})$  defined as
\begin{align*}
S_0 & = -p(0)\nabla g +  \int_0^{t_{0}} DM^\transp \nabla u \nabla p - (u-u_d)\nabla u_d -p\nabla f ,\\
S_1 & = 
\int_0^{t_{0}} -\nabla p\otimes M\nabla u  
-  \nabla u\otimes M^\transp\nabla p 
+ \left[M\nabla u\cdot \nabla p 
+ \frac{1}{2}(u - u_d)^2  -  pf \right] I_d,  
\end{align*}
and the adjoint $p$ is the solution of \eqref{eq:adjJ1}.
\end{proposition}
\begin{proof}
Using Definition \ref{def:transpose_third} of the transpose for third-order tensors, we obtain
$$(DM)\ta \nabla u \cdot \nabla p =  DM^\transp \nabla u \nabla p \cdot \ta.$$
Then, using Lemma \ref{lem01}, the tensor relations of Lemma \ref{lemma_tensor} and $\divv\ta = I_d : D\ta$ we get
\begin{align*}
 \mathcal{M}'(0,M)\nabla u \cdot \nabla p
 &= 
  \divv(\ta) M \nabla u \cdot \nabla p 
  -D\ta M \nabla u \cdot \nabla p
  - M D\ta^\transp \nabla u \cdot \nabla p \\
&  = 
  [-\nabla p\otimes M\nabla u  
-  \nabla u\otimes M^\transp\nabla p ]: D\ta
+ \left[M\nabla u\cdot \nabla p  \right] I_d : D\ta.
\end{align*}
Proceeding in a similar way for the other terms in $dJ_1(\Om)(\ta)$, we obtain \eqref{eq:sder_J1}.
The regularity of $S_1$ and $S_0$ is an immediate consequence of $u,p\in E$ and the regularity assumptions on $M,f,g$ and $u_d$.
\end{proof}
In a similar way we compute
\begin{align*} 
dJ_2(\Om)(\ta) & = \int_0^{t_{0}} \left\langle   \partial_{t} u, \divv(\ta)   p \right\rangle_{H^{-1}(\Om),H_{0}^{1}(\Om)}  
+\int_0^{t_{0}} \iom (\M'(0,M) + DM\ta)\nabla u \cdot \nabla p \\
&\quad + \iom u(0) p(0) \divv(\ta)   
- \int_0^{t_{0}} \iom p\nabla f\cdot \ta + pf\divv(\ta) 
- \iom p(0)\nabla g\cdot \ta + p(0)g\divv(\ta) \\
&\quad + \iom - (u(t_{0}) - u_d)\nabla u_d\cdot \ta + \frac{1}{2}(u(t_{0}) - u_d)^2 \divv(\ta).
\end{align*}
Using $u(0)=g$, this yields the following result, whose proof is similar to the proof of Proposition \ref{prop:3}.
\begin{proposition}
Let $\Om\in\mathds{P}(\mathcal{D})$ be  of class  $\C^{2}$, $\ta\in \C^2_c(\hold,\Rd)$, $f\in L^{2}(0,t_{0};H^{1}(\mathcal{D}))$, $g\in H^1(\mathcal{D})$, $u_{d}\in H^1(\mathcal{D})$, and $M\in\C^{1}([0,t_{0}]\times\mathcal{D},\R^{d\times d})$ 
be symmetric and uniformly positive definite.
Then, the distributed shape derivative of the cost functional \eqref{cost2_parabolic} subject to the constraint \eqref{eq:parabolicmain} is given by
\begin{align} \label{eq:sder_J2}
dJ_2(\Om)(\ta) & =  \int_0^{t_{0}} \left\langle   \partial_{t} u, \divv(\ta)   p \right\rangle_{H^{-1}(\Om),H_{0}^{1}(\Om)} +\iom S_0\cdot \ta + S_1 : D\ta,
\end{align}
with $S_0\in L^1(\Om, \R^{d})$, $S_1\in L^1(\Om, \R^{d\times d})$, defined as
\begin{align*}
S_0 & = -p(0)\nabla g  - (u(t_0) - u_d)\nabla u_d +  \int_0^{t_{0}} DM^\transp \nabla u \nabla p  -p\nabla f ,\\
S_1 & = \frac{1}{2}(u(t_0) - u_d)^2 I_d 
+  \int_0^{t_{0}} -\nabla p\otimes M\nabla u  
-  \nabla u\otimes M^\transp\nabla p 
+ \left[M\nabla u\cdot \nabla p 
  -  pf \right] I_d,  
\end{align*}
and the adjoint $p$ is the solution of \eqref{eq:adjJ2}.
\end{proposition}
\begin{remark}
In \cite[Section 3.4]{MR1215733} the material and shape derivatives of the solutions to parabolic problems with Neumann and Dirichlet conditions were investigated in the particular case $M=I_d$. 
The results of this section generalize the results of  \cite[Section 3.4.2]{MR1215733} and give a new perspective on the duality between the material derivative $\dot{u}$ and the adjoint $p$ in the parabolic case.
\end{remark}

\section{Comparison with other methods}\label{sec:aam}
\subsection{Comparison with the averaged adjoint method}
In this section we study the relation between Theorem~\ref{thm1} and the averaged adjoint method (AAM) introduced in \cite{MR3374631}; see also  \cite{MR3609755,MR3711067,2020arXiv200509011G,MR3535238,2018arXiv180300304S} for variations and extensions of the AAM. 
Let $\Om\in\mathds{P}(\mathcal{D})$, $\ta\in \C^{0,1}_c(\hold,\Rd)$, the associated flow
$\Tt:\hold\rightarrow \hold$ and  $\Om_{s}:=T_{s}(\Om)$ be defined as in Section \ref{section1a}.
Suppose that $E,F$ are Banach spaces with $F$ reflexive, that $\G$ has the form \eqref{G_lag} and that  
 $\A\in \C^1([0,s_1]\times E, F^*)$.
Recall that the shape functional is defined as 
$$ J(\Om_s) = \G(s,u^s,\psi) =  \langle  \A(s,u^s),\psi\rangle_{F^*,F} + \B(s,u^s) \text { for all }\psi\in F.$$
Although  the AAM is formulated with weaker assumptions in \cite{MR3374631},   we recall here a more compact version, adapted from \cite{MR3535238}, which is still quite general and easier to  compare with our approach. 
\begin{theorem}[Averaged adjoint method]\label{thm:AAE}
Assume that  for every $(s,\psi)\in [0,s_1] \times F$ we have that
 \begin{enumerate} 
\item[(H1)] the mapping $[0,1]\ni \eta\mapsto \G(s,\eta u^s +(1-\eta) u^0 ,\psi)$ is absolutely continuous;
\item[(H2)] the mapping $[0,1]\ni \eta\mapsto \langle\partial_\vp \G(s,\eta u^s +(1-\eta) u^0 ,\psi),\hat\vp\rangle_{E^*,E}$ belongs to $L^1(0,1)$ for every $\hat\vp\in E$;
\item[(H3)] there exists a unique averaged adjoint $p^s \in F$ solution of the averaged adjoint equation
\begin{equation}\label{averated_adj}
\int_0^1 \left\langle\partial_\vp \G(s,\eta u^s +(1-\eta) u^0,p^s),\hat\vp\right\rangle_{E^*,E}\, d\eta =0 \quad \forall \hat\vp\in E;
\end{equation}
\item[(H4)] we have
\begin{equation}\label{eq:DifferenceQuotientAvAdj}
 \lim_{s\searrow 0} \frac{\G(s, u^0 ,p^s)-\G(0, u^0 ,p^s)}{s}=\partial_s \G(0, u^0 , p^0). 
\end{equation}
\end{enumerate}
Then $J$ is shape-differentiable  and it holds that
 \begin{equation*}
 dJ(\Om)(\VV) = \partial_s \G(0, u^0 ,p^0).
 \end{equation*}
\end{theorem}
On the one hand, we observe that conditions (H1) and (H2) are stronger with respect to $\B$ than the conditions of Theorem \ref{thm1}, as they also implicitly require the absolute continuity of $[0,1]\ni \eta\mapsto \B(s,\eta u^s +(1-\eta) u^0)$ and the integrability of $[0,1]\ni \eta\mapsto \partial_\vp \B(s,\eta u^s +(1-\eta) u^0 ;\hat\vp)$.
On the other hand, the conditions of Theorem~\ref{thm1} are  stronger with respect to $\A$, as Theorem~\ref{thm1} requires $\A\in \C^{1}(\left[0,s_1\right]\times E,F^{*})$ whereas condition (H1) implicitely requires the absolute continuity of $[0,1]\ni \eta\mapsto \langle  \A(s,\eta u^s +(1-\eta) u^0),\psi\rangle_{F^*,F}$.
In any case, conditions (H1) and (H2) are usually easily verified in practice, so that  (H3) and (H4) are in fact the interesting conditions to compare with the conditions of Theorem~\ref{thm1}. 

For the sake of comparison we assume that the conditions of Theorem~\ref{thm1} and (H1)-(H3) of Theorem~\ref{thm:AAE} are satisfied.
Then, the averaged adjoint $p^s$ is the solution to
\begin{align}\label{aadjoint}
 \int_0^1  \langle  \partial_{\vp}\A(s, \eta u^s + (1-\eta) u^0)(\hat\vp), p^s \rangle_{F^*,F} \, d\eta 
 = - \int_0^1 \langle\partial_{\vp} \B(s,\eta u^s + (1-\eta) u^0),\hat\vp\rangle_{E^*,E}\, d\eta, \quad \forall \hat\vp\in E. 
\end{align}
Note that for $s=0$, the averaged adjoint $p^0$ is the solution to
\begin{equation}\label{eq:34}
\langle  \partial_{\vp}\A(0,u^0)(\hvp),p^0\rangle_{F^*,F} = - \langle\partial_\vp \B(0,u^0),\hvp\rangle_{E^*,E} \text{ for all }\hvp\in E,  
\end{equation}
therefore it coincides with the adjoint $p$ defined in \eqref{adj_eq}.

Now let us check  condition  (H4) of Theorem \ref{thm:AAE}. 
We compute
\begin{align*}
\lim_{s\searrow 0}\frac{\G(s,u^0,p^s) - \G(0,u^0,p^s)}{s}
& = 
\lim_{s\searrow 0}\frac{ \langle  \A(s,u^0), p^s\rangle_{F^*,F} -  \langle  \A(0,u^0), p^s\rangle_{F^*,F}}{s}
+ \lim_{s\searrow 0}\frac{\B(s,u^0) - \B(0,u^0)}{s}\\
& =\lim_{s\searrow  0}
\langle\frac{ \A(s,u^0) - \A(0,u^0)}{s},p^s\rangle_{F^*,F}
 + \partial_s \B(0,u^0).
\end{align*}
In practice, this limit is often computed by proving the strong convergence $(\A(s,u^0) - \A(0,u^0))/s\to\partial_s \A(0,u^0)$ in $F^*$, and  the weak convergence $p^s\rightharpoonup p^0$ in $F$, see for instance \cite{MR3609755,MR3436555, MR3374631}, which yields
\begin{align}\label{eq:limGs}
\lim_{s\searrow 0}\frac{\G(s,u^0,p^s) - \G(0,u^0,p^s)}{s}
& = \langle \partial_s \A(0,u^0),p^0\rangle_{F^*,F}
+ \partial_s \B(0,u^0)
= \partial_s \G(0,u^0,p^0),
\end{align}
and in  this case Theorem~\ref{thm1} can also be used.

However, the limit \eqref{eq:limGs} can also be obtained by proving the weak convergence $(\A(s,u^0) - \A(0,u^0))/s\rightharpoonup\partial_s \A(0,u^0)$ in $F^*$ and the strong convergence  $p^s\to p^0$ in $F$, and in this case Theorem~\ref{thm1} cannot be used.
A simple example illustrating this approach is presented in Section \ref{sec:kun}.

We gather from this comparison that when $\A\in \C^1([0,s_1]\times E, F^*)$, which usually  is a consequence of having sufficiently regular data, Theorem \ref{thm1} can be applied and  there is no need to introduce and study the averaged adjoint and to prove the weak convergence $p^s\rightharpoonup p^0$, which sometimes leads to lengthy proofs, particularly when $\A(s,\vp)$ is non-linear in $\vp$; see for instance \cite{MR3436555, LWY, MR3374631}. 
The other advantage is that the relations between the adjoint, the material derivative and the shape derivative of the cost functional appear clearly when using Theorem \ref{thm1}.
However, Theorem \ref{thm1} cannot be applied when $\A$ is less regular and is only weakly differentiable with respect to $s$, whereas the AAM can be applied in this case.
Thus, we conclude that the AAM is a very versatile approach which is well-suited for singular situations, as shown in \cite{MR3609755}, but for smoother scenarios the approach based on Theorem \ref{thm1} and the implicit function theorem is more straightforward.
We also mention  that the AAM was extended recently to even more singular frameworks such as nonsmooth cost functions \cite{MR3584578} and topological derivatives \cite{2020arXiv200509011G,2018arXiv180300304S}.

\subsection{Comparison with a variational approach to shape derivatives}\label{sec:kun}
It is interesting to revisit a counter-example that was presented in \cite[Section 3.5]{MR2434064}, where $\Om$ is assumed to be in $\R^3$ and of class $\C^{2,1}$.
We investigate how this simple example fares with Theorem \ref{thm1} and with the AAM.
Our setting is $\Om\in\mathds{P}(\mathcal{D})$, $E=F=H^1_0(\Om)$,  $E^*=F^*=H^{-1}(\Om)$, $\ta\in \C^1_c(\hold,\Rd)$, $f\in L^2(\hold)$, $f^s:=f\circ\Tt$. Define
\begin{align*}
\langle\mathcal{A}(s,\varphi),\psi\rangle_{F^{*},F} & :=\int_{\Omega}\mathcal{M}(s,I_{d})\nabla\varphi\cdot\nabla\psi-f^{s}\psi\xi(s),\\
\mathcal{B}(s,\varphi) & :=\int_{\Omega}\mathcal{M}(s,I_{d})\nabla\varphi\cdot\nabla\varphi.
\end{align*}
Here $\mathcal{A}$ is associated with the Poisson problem with Dirichlet conditions and $\mathcal{B}$ with the cost function
\[
J(\Omega)=\int_{\Omega}|\nabla u|^{2}.
\]
Then it is clear that 
$$\langle A\hat{\vp},\psi\rangle_{F^*,F}
:= \langle \partial_\vp\A(0,u)(\hat\vp ),\psi \rangle_{F^*,F} = 
\int_\Om \nabla \hvp\cdot \nabla\psi,$$ 
and  $A: E\to F^*$ is an isomorphism.
In this case we have $A = A^*$  so we immediatly get the following adjoint equation
$$\langle A^*p,\hvp\rangle_{E^*,E} = - \langle \bl(u),\hvp\rangle_{E^*,E}\quad \text{ for all }\hvp\in E,$$
which is
$$ \int_\Om \nabla \hvp\cdot \nabla p =  - 2\int_\Om  \nabla u\cdot\nabla\hvp\quad  \text{ for all }\hvp\in H^1_0(\Om).$$
This implies $p=-2u$.

First of all when $f$ is only in $L^2(\hold)$, then using \cite[Proposition 2.39]{MR1215733} we have that $s\mapsto f^s$ is weakly differentiable in $H^{-1}(\hold)$, but $s\mapsto f^s$ is not strongly differentiable in $H^{-1}(\hold)$; see the  counter-example in \cite[p. 73]{MR1215733}.
So in this case we cannot show that $\A\in \C^{1}([0,s_1]\times E,F^{*})$ and  Theorem \ref{thm1} cannot be applied.

Now if $f\in W^{1,q}(\hold)$ with $q>1$ and $d=2$, then $s\mapsto f^s$ is strongly differentiable in $L^q(\hold)$.
Let us define $F(s) \in H^{-1}(\hold)$ as
$$ F(s): H^1_0(\hold)\ni \psi \mapsto  \int_\hold \left(\frac{f^s \xi(s) - f}{s}  - \divv(f\theta)\right)  \psi .$$
Due to the Sobolev imbedding $H^1_0(\hold)\subset L^{q'}(\hold)$ for all $1\leq q'< \infty$ in two dimensions, we have $\psi\in L^{q'}$  for all $1\leq q'< \infty$.
Using H\"older's inequality, with $1/q + 1/q' = 1$, and  the fact that $s\mapsto f^s$ is strongly differentiable in $L^q(\hold)$ for $q>1$ we obtain
\begin{align*}
\| F(s) \|_ {H^{-1}(\hold)} 
& = \sup_{\|\psi\|_{ H^1_0(\hold)}=1}\left|  \int_\hold \left(\frac{f^s \xi(s) - f}{s}  - \divv(f\theta)\right)  \psi \right| \\
& \leq \underbrace{\sup_{\|\psi\|_{ H^1_0(\hold)}=1}\left( \int_\hold |\psi|^{q'}  \right)^{1/q'}}_{\leq C}
\underbrace{\left( \int_\hold\left|\frac{f^s \xi(s) - f}{s}  - \divv(f\theta)\right|^q   \right)^{1/q}}_{\to 0 \text{ as } s\to 0},
\end{align*}
where $C>0$ is the norm of the inclusion $H^1_0(\hold)\hookrightarrow L^{q'}(\hold)$.
Hence $\lim_{s\to 0}\| F(s) \|_ {H^{-1}(\hold)} = 0$ and we can show that  $\A\in \C^{1}([0,s_1]\times E,F^{*})$ in a similar way.
Thus, in this case  Theorem \ref{thm1} can be applied and  $u$ has a material derivative  $\dot u\in H^1_0(\Om)$ which is the unique solution of 
$$\langle A\dot u,\psi\rangle_{F^*,F} = -\langle L(u),\psi\rangle_{F^*,F}\quad \text{ for all }\psi\in F ,$$
which means 
$$ \int_\Om \nabla \dot u\cdot \nabla \psi 
=  - \int_\Om \M'(0,I_d)\nabla u\cdot \nabla \psi - \divv(f\ta) \psi   \quad \text{ for all }\psi\in H^1_0(\Om).$$

Now if $f\in W^{1,q}(\hold)$ with $q>1$ and $d=3$, then $s\mapsto f^s$ is strongly differentiable in $L^q(\hold)$ but we only have  $\psi\in L^{q'}$  for all  $1\leq q'\leq 6$.
Due to $1/q + 1/q' = 1$, this implies that we can prove $\lim_{s\to 0}\| F(s) \|_ {H^{-1}(\hold)} = 0$ and  $\A\in \C^{1}([0,s_1]\times E,F^{*})$ only for $q\geq 6/5$.
This shows that, for $f\in W^{1,q}(\hold)$ with $q\in (1,6/5)$ in three dimensions, we cannot apply Theorem~\ref{thm1} due to the condition $\A\in \C^{1}([0,s_1]\times E,F^{*})$. 

We actually obtain the same restriction $q\geq 6/5$ as in the counter-example  of \cite[Section 3.5]{MR2434064}.
However, the perspective in  \cite{MR2434064} is different:
it is shown that the shape derivative of $J(\Om)$ exists and can be computed with the method of \cite{MR2434064} for $f\in W^{1,q}(\hold)$, with $q\in (1,6/5)$, even though  the shape derivative $u'$ defined by $u':=\dot u - \nabla u\cdot \ta$ does not exist (in the sense that  it is not in $H^1(\Om)$).
Our approach shows that the lack of strong differentiability of $s\mapsto f^s$  precludes the application of the implicit function theorem in Theorem~\ref{thm1}, and also the existence of a strong material derivative $\dot u$. 
In fact, in this case it is known that  the weak material derivative $\dot u$ exists, see \cite[Corollary 2.81]{MR1215733}.
This indicates that a central issue to prove the shape differentiability of $J(\Om)$ is the question of the weak or strong  material derivative $\dot u$, rather than the existence of the shape derivative $u'$.

To complete the comparison we discuss how the AAM can be applied to this example.
We assume that $\Om\in\mathds{P}(\mathcal{D})$,  $f\in W^{1,q}(\hold)$ with $q\in (1,6/5)$ and $d=3$.
First of all, it is easy to compute the variational formulation corresponding to the definition  \eqref{aadjoint} of the  averaged adjoint $p^s$:
$$ \int_\Om  \nabla p^s\cdot \nabla\hvp  =  - \iom \M(s,I_d)  \nabla (u^s+u^0)\cdot \nabla \hvp\quad  \text{ for all }\hvp\in H^1_0(\Om).$$
Taking the difference between the above equation and the same equation at $s=0$, and choosing the test function $\hvp = p^s -p^0$ leads to the estimate
\begin{equation}\label{eq:cvps}
 \| \nabla (p^s -p^0) \|_{L^2(\Om)} \leq \| \M(s,I_d)  \nabla (u^s+u^0) - 2\nabla u^0 \|_{L^2(\Om)}. 
\end{equation}

Now we provide a similar estimate for the gradient of $u^s$.
The solution $u^s$ satisfies
$$0 = \int_\Om \M(s,I_d) \nabla u^s\cdot \nabla\psi  -  f^s \psi \xi(s)
=\int_\Om \nabla u^s\cdot \nabla\psi + (\M(s,I_d)  - I_d)\nabla u^s\cdot \nabla\psi -  f^s \psi \xi(s) .$$
Then, taking the difference between the above equation and the same equation at $s=0$, using $\M(0,I_d)=I_d$, and choosing the test function $\psi = u^s -u_0$ we obtain
$$0 =\int_\Om \nabla (u^s -u^0)\cdot \nabla (u^s -u_0) + (\M(s,I_d)  - I_d)\nabla u^s\cdot \nabla (u^s -u_0)-   (u^s -u_0)(f^s\xi(s)-f) .$$
Then, using the Poincaré inequality one obtains
\begin{align*}
 \| \nabla (u^s -u^0) \|^2_{L^2(\Om)} & \leq \| \M(s,I_d)  -I_d \|_{L^\infty(\Om)}  \| \nabla u^s \|_{L^2(\Om)} \| \nabla (u^s -u^0) \|_{L^2(\Om)}  \\
&\quad + C_\Om\| f^s \xi(s) -f \|_{L^2(\Om)} \| \nabla (u^s -u^0) \|_{L^2(\Om)}, 
\end{align*}
where $C_\Om>0$ is the constant coming from the Poincaré inequality.
Finally, dividing by  $\| \nabla (u^s -u^0) \|_{L^2(\Om)}$ we obtain  the estimate:
$$ \| \nabla (u^s -u^0) \|_{L^2(\Om)} \leq \| \M(s,I_d)  -I_d \|_{L^\infty(\Om)}  \| \nabla u^s \|_{L^2(\Om)}  + C_\Om\| f^s \xi(s) -f \|_{L^2(\Om)}.$$
Using the convergences of $\M(s,I_d) $, $\xi(s)$ and $f^s$ and the uniform boundedness of  $\| \nabla u^s \|_{L^2(\Om)}$ we obtain the strong convergence $u^s \to u$ in $H^1_0(\Om)$, which in turn yields the strong convergence $p^s \to p$ in $H^1_0(\Om)$ via \eqref{eq:cvps}.

Then we have
\begin{align*}
\lim_{s\searrow  0} \langle\frac{ \A(s,u^0) - \A(0,u^0)}{s},p^s\rangle_{F^*,F}
&=
\lim_{s\searrow  0} \iom  \frac{(\M(s,I_d)  -I_d)\nabla u^0\cdot \nabla p^s}{s}  - \frac{ (f^s \xi(s) - f) p^s}{s}\\
&= \iom \M'(0,I_d)\nabla u\cdot \nabla p - \divv(\ta) fp \, dx
- \left\langle \nabla f\cdot\ta, p\right\rangle _{H^{-1}(\Om)\times H_{0}^{1}(\Om)}, 
\end{align*}
where we have used the weak differentiability of   $s\mapsto f^s$ in $H^{-1}(\hold)$ and the strong convergence $p^s \to p$ in $H^1_0(\Om)$.
Thus we conclude that in the case $f\in W^{1,q}(\hold)$ with $q\in (1,6/5)$ and $d=3$, the AAM can be applied even though Theorem~\ref{thm1} cannot, at the cost of proving the  strong convergence of the averaged adjoint $p^s \to p$ in $F$.
We also observe that the weak convergence $(\A(s,u^0) - \A(0,u^0))/s\rightharpoonup\partial_s \A(0,u^0)$ in $F^*=H^{-1}(\Om)$ that we have used to apply the AAM is directly related to the existence of  the weak material derivative $\dot u$, which is proven in \cite[Corollary 2.81]{MR1215733}.

Finally, we compute the shape derivative of $J(\Om)$, $\Om\in\mathds{P}(\mathcal{D})$, in the case where Theorem~\ref{thm1} can be applied, i.e. for $f\in W^{1,q}(\hold)$ with $q\geq 6/5$:
$$ dJ(\Om)(\ta) =\langle L(u),p \rangle_{F^*,F} + \partial_s \B(0,u),$$
so we obtain, using $p=-2u$,
$$ dJ(\Om)(\ta) = \int_\Om \M'(0,I_d)\nabla u\cdot \nabla p - \divv(f\ta) p + \M'(0,I_d)\nabla u\cdot \nabla u
= \int_\Om 2 \divv(f\ta) u - \M'(0,I_d)\nabla u\cdot \nabla u .$$
The shape derivative can also be written in tensorial form as
\begin{align*} 
dJ(\Om)(\ta) & = \iom S_0 \cdot \ta + S_1 : D\ta.
\end{align*}
with  $S_0 = 2u\nabla f$ and  $S_1 = 2\nabla u\otimes\nabla u   + (2 fu - |\nabla u|^2) I_d$.
When $\Om$ is of class $\C^1$, using \cite[Proposition 4.3]{MR3535238} we immediately obtain the following boundary expression, also known as Hadamard formula:
\begin{align*} 
dJ(\Om)(\ta) & = \int_{\partial\Om} (S_1 n\cdot n) \ta\cdot n =  \int_{\partial\Om} |\dn u|^2 \ta\cdot n.
\end{align*}
which yields the same formula as in \cite[Section 3.5]{MR2434064}.

We conclude that when applicable, Theorem~\ref{thm1} allows us to quickly obtain the material derivative equation, the adjoint equation,   the distributed shape derivative when $\Om$ has low regularity,  and the corresponding boundary expression when $\Om$ is $\C^1$; compare with the development in \cite[Section 3.5]{MR2434064}.
Nevertheless, the condition $\A\in \C^{1}([0,s_1]\times E,F^{*})$ imposes a restriction on the regularity of the right-hand side $f$.
Also, this issue cannot be resolved by using  the framework of Section \ref{sec:distribution2}, i.e. by using a larger space $E$, since the cost functional cannot depend on $\nabla u$ in Section \ref{sec:distribution2}.
Based on this observation, an interesting direction for further research would be the application of a  weak form of the implicit function theorem in the spirit of  \cite[Theorem 4.1]{MR1452889} or  \cite[Theorem 2.1]{MR3165308}, which does not require $\A\in \C^{1}([0,s_1]\times E,F^{*})$, and would allow to exploit the weak differentiability of  $s\mapsto f^s$ in $H^{-1}(\hold)$.

\vspace{0.5cm}
\noindent 
{\bf Acknowledgments. } Antoine Laurain gratefully acknowledges the support of  the Brazilian National Council for Scientific and Technological Development  (Conselho Nacional de Desenvolvimento Cient\'ifico e Tecnol\'ogico - CNPq) through the process: 408175/2018-4 ``Otimiza\c{c}\~ao de forma n\~ao suave e controle de problemas de fronteira livre'', and through the program  ``Bolsa de Produtividade em Pesquisa - PQ 2018'', process: 304258/2018-0.
Pedro T. P. Lopes was partially supported by grant \#2019/15200-1, São Paulo Research Foundation (FAPESP). Jean C. Nakasato is supported by CAPES - INCTmat grant 465591/2014-0.

\section{Appendix}\label{sec:appendix}
\subsection{Second-order chain rule}
Here we prove the following result
(we also refer to \cite[Lemma 2.62]{MR1215733} for other formulas for transport of differential operators).
\begin{lemma}\label{lemma_second_order_chain}
Let $\psi\in H^2(\hold)$ and $\ta\in \C^2_c(\hold,\Rd)$. Then we have 
$s\mapsto[\Delta(\psi\circ T_s^{-1})]\circ T_s \in\C^1([0,s_1], L^2(\hold))$
and
\begin{align}
\label{der_D2} \partial_s ([D^2(\psi\circ T_s^{-1})]\circ T_s)|_{s=0} 
& = - D\theta^{\transp} D^2\psi -   D^2\psi D\theta 
- D^2\theta^{\transp}\nabla\psi, \\
\label{der_lap}  \partial_s ([\Delta(\psi\circ T_s^{-1})]\circ T_s)|_{s=0} 
& = -2 D^2\psi : D\theta - (\Delta \theta) \cdot \nabla\psi.
\end{align} 
\end{lemma}
Let $\psi:\R^d\to\R$, $F\in \C^2(\hold,\Rd)$, $G\in \C^2(\hold,\Rd)$, and introduce the first and second-order differentials $d(\psi\circ G): \R^d \to \mathcal{L}(\R^d,\R)$ and $d^2(\psi\circ G): \R^d \to \mathcal{L}(\R^d,\mathcal{L}(\R^d,\R))$.
The chain rule yields, with $H\in\R^d$,
\begin{equation*}
d(\psi\circ G)(x)(H) = d\psi(G(x))(dG(x)(H)),
\end{equation*}
and 
\begin{align*}
d^2(\psi\circ G)(x)(H_1,H_2) & = d(d(\psi\circ G)(x)(H_1))(H_2) = d(d\psi(G(x))(dG(x)(H_1)))(H_2),\\
& = d^2\psi(G(x))(dG(x)(H_1),dG(x)(H_2))
+ d\psi(G(x)) (d^2 G(x)(H_1,H_2)),
\end{align*}
with $H_1,H_2\in\R^d$, and where
$d^2 G: \R^d \to \mathcal{L}(\R^d,\mathcal{L}(\R^d,\R^d))$.
Thus $d^2 G(x)\in \mathcal{L}(\R^d,\mathcal{L}(\R^d,\R^d)) \cong \R^{d\times d\times d}$ and we can identify $d^2 G(x)$ with a tensor of order $3$ denoted by $D^2G$.
To be more precise we have
$$ d^2 G(x)(H_1,H_2) = D^2G(x)H_1 H_2\in\R^d,$$
and the entries of the tensor $D^2 G$ are $(D^2 G)_{ijk} = (\partial_{jk} G_i)$.
By symmetry of second-order derivatives, we have the property
$$ D^2G(x)H_1 H_2 = D^2G(x)H_2 H_1\in\R^d,$$
and we say that $D^2G(x)$ is {\it right-side symmetric}; see \cite{qi2017transposes}.

Thus, we have (see Section \ref{section1a0} for the definition of the transpose of a third-order tensor),
\begin{align*}
d\psi(G(x)) (d^2 G(x)(H_1,H_2)) & =  \nabla\psi (G(x))  D^2G(x)H_1 H_2
= \nabla\psi (G(x))  D^2G(x)H_2 H_1\\
& =  H_2  D^2G(x)^\transp H_1 \nabla\psi (G(x))  
=  (D^2G(x)^{\transp} \nabla\psi (G(x)) ) H_1 \cdot H_2 .
\end{align*} 
Note that $D^2G(x)^{\transp} \nabla\psi (G(x))$ is a symmetric matrix.

For the other term we have
\begin{align*}
d^2\psi(G(x))(dG(x)(H_1),dG(x)(H_2))
& = 
D^2\psi(G(x)) DG(x)H_1 \cdot DG(x)H_2\\
& =
DG(x)^\transp D^2\psi(G(x)) DG(x)H_1 \cdot H_2.
\end{align*}
Thus we have found
\begin{align*}
d^2(\psi\circ G)(x)(H_1,H_2) & = DG(x)^\transp D^2\psi(G(x)) DG(x)H_1 \cdot H_2
+ (D^2G(x)^{\transp} \nabla\psi (G(x)) ) H_1 \cdot H_2,
\end{align*}
so we can identify $d^2(\psi\circ G)(x)$ with the second-order tensor
$$ D^2(\psi\circ G)(x) = DG(x)^\transp D^2\psi(G(x)) DG(x) + D^2G(x)^{\transp} \nabla\psi (G(x)),$$
and we immediately get
\begin{align*}
\Delta(\psi\circ G)(x) & = \tr( D^2(\psi\circ G)(x)) = \tr(DG(x)^\transp D^2\psi(G(x)) DG(x)) + \tr(D^2G(x)^{\transp} \nabla\psi (G(x))). 
\end{align*}

Now suppose that $G = F^{-1}$, then $ G\circ F(x) = x$ and $D(G\circ F)(x) = I_d$ which yields 
$$DG(F(x)) DF(x) = I_d$$ 
or equivalently
$ DG\circ F = DF^{-1}$.
Using this property we compute
\begin{align}\label{177}
\begin{split}
 [D^2(\psi\circ G)]\circ F &= (DG^\transp D^2\psi \circ G DG)\circ F + (D^2G^{\transp} \nabla\psi \circ G)\circ F\\
 & = DF^{-\transp} D^2\psi DF^{-1} +  D^2G^{\transp}\circ F \nabla\psi.
\end{split}
\end{align}
In a similar way we have
$
d(G\circ F)(x)(H) = d G(F(x))(dF(x)(H)),
$
and then 
\begin{align*}
0 = d^2(G\circ F)(x)(H_1,H_2) 
& = d^2 G(F(x))(dF(x)(H_1),dF(x)(H_2))
+ dG(F(x)) (d^2 F(x)(H_1,H_2)).
\end{align*}
Note that $d^2(G\circ F) \equiv 0$ since we have $(G\circ F)(x) =x$. 
Thus we obtain
\begin{align*}
d^2 G(F(x))(dF(x)(H_1),dF(x)(H_2)) & 
= - dG(F(x)) (d^2 F(x)(H_1,H_2)),
\end{align*}
or equivalently
\begin{align*}
D^2 G(F(x))[DF(x)(H_1)] [DF(x)(H_2)] & 
= - DG(F(x)) (D^2 F(x)H_1 H_2).
\end{align*}
Alternatively, introducing $\hat H_1 :=DF(x)(H_1)$ and  $\hat H_2 :=DF(x)(H_2)$, we can write this relation as
\begin{align*}
D^2 G \circ F  \hat H_1 \hat  H_2 & 
= - DG\circ F  (D^2 F[DF^{-1}(\hat H_1)] [DF^{-1}(\hat  H_2)])\\
& = - DF^{-1}  (D^2 F[DF^{-1}(\hat H_1)] [DF^{-1}(\hat  H_2)]).
\end{align*}

Now we identify $F = T_s$ and $G = T_s^{-1}$,
this yields
\begin{align}\label{D2Ts-1}
D^2 (T_s^{-1}) \circ T_s  \hat H_1 \hat  H_2 
& = - DT_s^{-1}  (D^2 T_s[DT_s^{-1}(\hat H_1)] [DT_s^{-1}(\hat  H_2)]).
\end{align}
Using \eqref{177} we obtain
\begin{align}\label{eq:222}
[D^2 (\psi\circ T_s^{-1})]\circ T_s & = DT_s^{-\transp} D^2\psi DT_s^{-1}
+ D^2(T_s^{-1})^{\transp}\circ T_s \nabla\psi ,
\end{align}
and then
$
[\Delta(\psi\circ T_s^{-1})]\circ T_s = \tr(DT_s^{-\transp} D^2\psi DT_s^{-1})
+ \tr(D^2(T_s^{-1})^{\transp}\circ T_s \nabla\psi )
$.
Using $\psi\in H^2(\hold)$,  \eqref{D2Ts-1} and  Lemma \ref{lem01}, this shows that
$$s\mapsto[\Delta(\psi\circ T_s^{-1})]\circ T_s \in\C^1([0,\tz], L^2(\hold)).$$

We have then $\partial_s DF|_{s=0} = D\theta$ and $\partial_s DG|_{s=0} = -D\theta$.
Note that since $F(x)|_{s=0} = T_0(x)=x$, we have $DF(x)|_{s=0} = DT_0(x)= I_d$ and $D^2 F(x)|_{s=0} = D^2 T_0(x)= 0$. 
Using this we obtain
\begin{align*}
\partial_s (D^2 G \circ F)|_{s=0}  \hat H_1 \hat  H_2 & 
= [- DF^{-1}   (\partial_s  (D^2 F)[DF^{-1}(\hat H_1)] [DF^{-1}(\hat  H_2)])]|_{s=0}
= - D^2\theta \hat H_1 \hat H_2.
\end{align*}
Now, differentiating \eqref{eq:222} we get
\begin{align*}
\partial_s ([D^2 (\psi\circ G)]\circ F)|_{s=0} & = \partial_s (DF^{-\transp} D^2\psi DF^{-1})|_{s=0}
+ \partial_s (D^2G^{\transp}\circ F \nabla\psi )|_{s=0}\\
& = - D\theta^{\transp} D^2\psi -   D^2\psi D\theta 
- D^2\theta^{\transp}\nabla\psi ,
\end{align*}
and then
\begin{align*}
\partial_s ([\Delta(\psi\circ G)]\circ F)|_{s=0} 
& = -\tr( D\theta^{\transp} D^2\psi) - \tr(  D^2\psi D\theta) 
- \tr(D^2\theta^{\transp}\nabla\psi )\\
& = -2 D^2\psi : D\theta - \tr(D^2\theta^{\transp}\nabla\psi ).
\end{align*}
Finally, using Einstein summation convention  and $(D^2 \theta^{\transp})_{ijk} = (\partial_{ij}^2 \theta_k)$, we compute
$D^2\theta^{\transp}\nabla\psi 
= \partial^2_{ij}\theta_k  \partial_k\psi$.
This yields
\begin{equation}\label{eq:55}
\tr(D^2\theta^{\transp}\nabla\psi ) = \partial^2_{ii}\theta_k \partial_k\psi = (\Delta \theta) \cdot \nabla\psi, 
\end{equation}
and completes the proof.

\subsection{Proof of Proposition \ref{thm2}}
In order to prove Proposition \ref{thm2} it is enough to show that
$\partial_{s}\B$ and $\partial_{\vp}\B$ exist and are continuous \cite[Theorem 4.3]{ambrosetti1995primer}.
\begin{proposition}
\label{eq:prop4} The function $\partial_{s}\mathcal{B}:[0,s_{1}]\times L^{2}(\Omega)\to\R$
exists and is continuous. It is given by 
\[
\partial_{s}\mathcal{B}(s,\varphi)=\int_{\Omega}\nabla_{x}\mathcal{F}(T_{s}(x),\varphi(x))\cdot\partial_s T_{s}(x)\xi(s)+\mathcal{F}(T_{s}(x),\varphi(x))\xi'(s).
\]
\end{proposition}
\begin{proof}
The argument is similar to \cite[Section 1.3.2]{badiale2010semilinear}.
We know that $\left(s,x,r\right)\mapsto\mathcal{F}(T_{s}(x),r)\xi(s)$
is a $\mathcal{C}^{1}$ function. 
For $\vp\in L^{2}(\Omega)$ we consider
\[
G(s,h,x):=\frac{\mathcal{F}(T_{s+h}(x),\varphi(x))\xi(s+h)-\mathcal{F}(T_{s}(x),\varphi(x))\xi(s)}{h}.
\]
By the chain rule, it is clear that, for almost every $x\in\Omega$,
we have 
\[
\lim_{h\to0}G(s,h,x)=\nabla_{x}\mathcal{F}(T_{s}(x),\varphi(x))\cdot\partial_s T_{s}(x)\xi(s)+\mathcal{F}(T_{s}(x),\varphi(x))\xi'(s).
\]
The mean value theorem and our assumptions \eqref{F1.1}-\eqref{F1.2} on $\mathcal{F}$ also imply that
\[
|G(s,h,x)|\le c_0+c_1\varphi(x)^{2},
\]
for some positive constants $c_0$ and $c_1$.
As $c_0+c_1\varphi(x)^{2}$ is integrable, Lebesgue's dominated convergence
theorem implies that 
\[
\begin{aligned} & \lim_{h\to0}\int_{\Omega}\frac{\mathcal{F}(T_{s+h}(x),\varphi(x))\xi(s+h)-\mathcal{F}(T_{s}(x),\varphi(x))\xi(s)}{h}\\
 & \quad =\int_{\Omega}\nabla_{x}\mathcal{F}(T_{s}(x),\varphi(x))\cdot\partial_s T_{s}(x)\xi(s)+\mathcal{F}(T_{s}(x),\varphi(x))\xi'(s).
\end{aligned}
\]
This shows that $\partial_{s}\mathcal{B}$ exists. 

In order to prove
its continuity at an arbitrary point $(s_{0},\varphi_{0})$ in $[0,s_{1}]\times L^{2}(\Omega)$,
we consider an arbitrary sequence $\left((s_{j},\varphi_{j})\right)_{j}$
that converges to $(s_{0},\varphi_{0})$. The continuity follows by
showing that there is always a subsequence $\left((s_{j_{k}},\varphi_{j_{k}})\right)_{k}$
such that $\partial_{s}\mathcal{B}(s_{j_{k}},\varphi_{j_{k}})$ converges
to $\partial_{s}\mathcal{B}(s_{0},\varphi_{0})$.
Such a subsequence can be found recalling that we can always find
a subsequence and a function $\sigma\in L^2(\Om)$ (\cite[Theorem 3.12]{rudin2006real})
such that $\lim_{k\to\infty}\varphi_{j_{k}}(x)=\varphi_{0}(x)$, a.e. $x\in\Omega$, and $|\varphi_{j_{k}}(x)|\le|\sigma(x)|$ for almost every $x\in\Omega$
and all $k\in\mathds{N}_{0}$.

These conditions imply that, for almost every $x\in\Omega$, 
\[
\begin{aligned} & \lim_{k\to\infty}\left(\nabla_{x}\mathcal{F}(T_{s_{j_{k}}}(x),\varphi_{j_{k}}(x))\cdot\partial_s T_{s_{j_{k}}}(x)\xi(s_{j_{k}})+\mathcal{F}(T_{s_{j_{k}}}(x),\varphi_{j_{k}}(x))\xi'(s_{j_{k}})\right)\\
 & \quad =\nabla_{x}\mathcal{F}(T_{s_{0}}(x),\varphi_{0}(x))\cdot\partial_s T_{s_{0}}(x)\xi(s_{0})+\mathcal{F}(T_{s_{0}}(x),\varphi_{0}(x))\xi'(s_{0})
\end{aligned}
\]
and
\[
\begin{aligned} & \left|\nabla_{x}\mathcal{F}(T_{s_{j_{k}}}(x),\varphi_{j_{k}}(x))\cdot\partial_s T_{s_{j_{k}}}(x)\xi(s_{j_{k}})+\mathcal{F}(T_{s_{j_{k}}}(x),\varphi_{j_{k}}(x))\xi'(s_{j_{k}})\right|\\
 &\quad \le c_0+c_1\varphi_{j_{k}}(x)^{2}\le c_0+c_1\sigma(x)^{2}.
\end{aligned}
\]
Then, Lebesgue's dominated convergence theorem implies that 
\[
\begin{aligned} & \lim_{k\to\infty}\int_{\Omega}\nabla_{x}\mathcal{F}(T_{s_{j_{k}}}(x),\varphi_{j_{k}}(x))\cdot\partial_s T_{s_{j_{k}}}(x)\xi(s_{j_{k}})+\mathcal{F}(T_{s_{j_{k}}}(x),\varphi_{j_{k}}(x))\xi'(s_{j_{k}})\\
 &\quad  =\int_{\Omega}\nabla_{x}\mathcal{F}(T_{s_{0}}(x),\varphi_{0}(x))\cdot\partial_s T_{s_{0}}(x)\xi(s_{0})+\mathcal{F}(T_{s_{0}}(x),\varphi_{0}(x))\xi'(s_{0}).
\end{aligned}
\]
That is,
$
\lim_{k\to\infty}\partial_{s}\mathcal{B}(s_{j_{k}},\varphi_{j_{k}})=\partial_{s}\mathcal{B}(s_{0},\varphi_{0})$.
\end{proof}

\begin{proposition}\label{prop:dphi_b}
The function $\partial_{\vp}\B:[0,s_{1}]\times L^{2}(\Om)\to L^{2}(\Om)$
exists and is continuous. Moreover
\[
\langle\partial_{\vp}\B(s,\vp),\hat{\vp}\rangle_{L^2(\Om)^*,L^2(\Om)}=\int_{\Om}\partial_r \mathcal{F}(T_{s}(x),\vp(x))\hat{\vp}(x)\xi(s).
\]
\end{proposition}

\begin{proof}
In view of \cite[Theorem 2.6]{djairolectures}, our assumptions for the function $\mathcal{F}$ imply that the function $\mathcal{S}$ defined below is $\C^{1}$:
\[
\vp\in L^{2}(\Om)\overset{\mathcal{S}}{\mapsto}\mathcal{F}(T_{s}(x),\vp(x))\xi(s)\in L^{1}(\Om).
\]
As $\psi\in L^{1}\overset{\mathcal{T}}{\mapsto}\int_{\Om}\psi$
is continuous and linear, we conclude that $\vp\mapsto \B(s,\vp)$
is $\C^{1}$ as it is equal to $\mathcal{T}\circ\mathcal{S}$. The
expression of the derivative also follows from \cite[Theorem 2.6]{djairolectures}.  Its continuity follows using the same arguments of Proposition \ref{eq:prop4}.
\end{proof}
%

\bibliographystyle{abbrv}
\bibliography{shape_lagrangian}
\end{document}